\documentclass[11 pt, oneside, reqno]{amsart}
\usepackage[top=1.5in,bottom=1in,left=1in,right=1in]{geometry}
\usepackage{amsmath,amssymb,amsthm,amsopn,extarrows,accents,faktor,enumitem,stmaryrd,mathtools}
\usepackage[normalem]{ulem}  
\usepackage[linktocpage=true]{hyperref}
\usepackage[dvipsnames]{xcolor}
\usepackage[normalem]{ulem}

\hypersetup{
    colorlinks = true,
    allcolors = {blue}}
\usepackage[dvipsnames]{xcolor}

\usepackage{setspace}
\usepackage{tikz-cd,tikz}
\usetikzlibrary{matrix,arrows,decorations.pathmorphing}

\newtheorem{lemma}[equation]{Lemma}
\newtheorem{theorem}[equation]{Theorem}
\newtheorem*{theorem*}{Theorem}
\newtheorem{proposition}[equation]{Proposition}
\newtheorem{corollary}[equation]{Corollary}

\theoremstyle{definition}
\newtheorem{definition}[equation]{Definition}
\newtheorem{remark}[equation]{Remark}
\newtheorem{example}[equation]{Example}

\numberwithin{equation}{section}

\newcommand{\hooklongrightarrow}{\lhook\joinrel\longrightarrow}

\DeclareMathOperator{\im}{im}
\newcommand{\nubar}{\overline{\nu}}
\newcommand{\mubar}{\overline{\mu}}

\newcommand{\phibar}{\overline{\varphi}}
\newcommand{\pbar}{\overline{p}}
\newcommand{\Xbar}{\overline{X}}
\newcommand{\TM}{\mathbb{T}M}

\newcommand{\T}{\mathbb{T}}
\newcommand{\G}{\mathcal{G}}
\newcommand{\HH}{\mathcal{H}}

\renewcommand{\O}{\mathcal{O}}

\newcommand{\s}{\subset}
\renewcommand{\L}{\mathcal{L}}
\newcommand{\too}{\longrightarrow}
\newcommand{\mtoo}{\longmapsto}

\newcommand{\C}{\mathbb{C}}

\newcommand{\g}{\mathfrak{g}}

\renewcommand{\t}{\mathfrak{t}}

\renewcommand{\c}{\mathfrak{c}}
\newcommand{\h}{\mathfrak{h}}
\newcommand{\z}{\mathfrak{z}}

\newcommand{\uu}{\mathfrak{u}}
\newcommand{\p}{\mathfrak{p}}
\newcommand{\Z}{\mathfrak{Z}}
\renewcommand{\l}{\mathfrak{l}}
\newcommand{\bh}{\overline{\h}}
\newcommand{\bH}{\overline{H}}

\newcommand{\D}{\mathbf{D}}
\DeclareMathOperator{\Id}{Id}
\DeclareMathOperator{\Ad}{Ad}
\DeclareMathOperator{\ad}{ad}
\newcommand{\sll}[1]{\mkern-4mu\mathbin{/\mkern-5mu/}_{\mkern-4mu{#1}}}
\newcommand{\tto}{{\;\substack{\too\\[-1em] \too}\;}}
\newcommand{\sss}{\mathsf{s}}
\newcommand{\ttt}{\mathsf{t}}
\DeclareMathOperator{\Lie}{Lie}
\newcommand{\Gad}{G_{\ad}}
\newcommand{\tmu}{\widetilde{\mu}}
\newcommand{\tm}{\widetilde{m}}
\newcommand{\tn}{\widetilde{n}}

\newcommand{\ts}{\tilde{s}}

\title{Reduction along strong Dirac maps}
\author[Ana B\u{a}libanu]{Ana B\u{a}libanu}
\address[Ana B\u{a}libanu]{Department of Mathematics, Louisiana State University, 494 E. Parker Blvd., Baton Rouge, LA  70803, USA}
\email{ana@math.lsu.edu}
\author[Maxence Mayrand]{Maxence Mayrand}
\address[Maxence Mayrand]{D\'epartement de math\'ematiques, Universit\'e de Sherbrooke, 2500 Bd de l'Universit\'e, Sherbrooke, QC, J1K 2R1, Canada}
\email{maxence.mayrand@usherbrooke.ca}
\date{}\date{}

\makeatletter
\def\l@subsection{\@tocline{2}{0pt}{3pc}{5pc}{}}
\makeatother

\begin{document}
\maketitle

\begin{abstract}
We develop a general procedure for reduction along strong Dirac maps, which are a broad generalization of Poisson momentum maps. We recover a large number of familiar constructions in Poisson and quasi-Poisson geometry, and we introduce new examples of Poisson, quasi-Poisson, and Dirac reduced structures. In particular, we obtain quasi-Poisson analogues of several classes of spaces that are studied in geometric representation theory.
\end{abstract}

\tableofcontents

\section*{Introduction}
If $\g$ is a Lie algebra whose dual is equipped with the Kostant--Kirillov--Souriau Poisson structure, any Poisson map $\Phi:(M,\pi_M)\too(\g^*,\pi_{KKS})$ induces an infinitesimal action of $\g$ on $M$ via
\begin{align*}
\g&\too TM\\
\xi&\mtoo \pi_M^\#(\Phi^*\xi).
\end{align*}
This action is Hamiltonian, and $\Phi$ is its moment map. According to the Marsden--Weinstein Theorem \cite{mar.wei:74}, we may use this moment map to ``reduce'' the symmetries of the manifold $M$ at any point $x$ of $\g^*$ which is a regular value by taking the quotient 
\[\Phi^{-1}(x)/\g_x\]
with respect to the action of the isotropy Lie algebra $\g_x$. If this quotient is a manifold, it inherits a Poisson structure from $M$.

If we replace the space $\g^*$ by an arbitrary Poisson manifold $(X,\pi_X)$, the graph $L$ of $\pi_X^\#$ is a Lie algebroid. Any Poisson map $\Phi:(M,\pi_M)\too (X,\pi_X)$ induces a Lie algebroid action of $L$ on $M$, and for any point $x$ of $X$ the isotropy Lie algebra $L_{x}$ acts on the fiber of $\Phi$ above $x$. If $x$ is a regular value and the reduced space
\[\Phi^{-1}(x)/L_x\]
is a manifold, reduction equips it with a natural Poisson structure.

More generally, when the space $(X,L_X)$ is Dirac, the correct analogue of a moment map is a \emph{strong Dirac map}, a notion first studied by Bursztyn and Crainic \cite{bur.cra:05}. A strong Dirac map is a forward-Dirac map $\Phi:(M,L_M)\too (X,L_X)$ which restricts to an isomorphism between the kernels of the Dirac structures. Such a map induces an action of the Lie algebroid $L_X$ on $M$, and at any regular value $x$ of $\Phi$ we can form the reduced space
\[\Phi^{-1}(x)/L_{X,x}.\]
If it is a manifold, the Dirac pullback of $L_M$ to the fiber $\Phi^{-1}(x)$ descends to a Poisson structure on this reduced space.

{We develop a generalization of this procedure in which reduction takes place along a submanifold and the quotient is by the action of a Lie algebroid or Lie groupoid. Concretely, for any Dirac manifold $(Z,L_Z)$ we introduce the notion of a \emph{$Z$-level} or \emph{relative reduction level} of $X$---a generalized submanifold $(S,\gamma)$ of $X\times \overline{Z}$ which satisfies certain compatibility conditions. Each $Z$-level is equipped with a \emph{stabilizer subalgebroid}
\[B_{S,\gamma}\subset L_X,\]
and we prove the following main theorem, as Theorem \ref{main}.}

{\begin{theorem*}
Let $\Phi:(M,L_M)\too (X,L_X)$ be a strong Dirac map, and suppose that $(S,\gamma)$ is a $Z$-level of $X$ which meets $\Phi$ cleanly. If the quotient 
\[(M\times_XS)/B_{S,\gamma}\]
is a manifold, it inherits from $M$ a natural Dirac structure, and the second projection descends to a strong Dirac map
\[\pbar:(M\times_XS)/B_{S,\gamma}\too Z.\]
\end{theorem*}}

\noindent In other words, the symmetries induced on $M$ by the action of $L_X$ are reduced to $L_Z$-symmetries on the quotient space $(M\times_XS)/B_{S,\gamma}$. In the special case when $Z$ is a point, all symmetries are removed and the resulting Dirac structure on the reduced space is Poisson{---reduction levels in this case can therefore be viewed as pre-Poisson submanifolds of the Dirac manifold $X$.}

When $(M, L_M)$ is presymplectic and $Z$ is a point, we also give an interpretation of this theorem in shifted symplectic geometry \cite{pan.toe.vaq.vez:13} as an intersection of two Lagrangians in a 1-shifted symplectic stack. Namely, if $(X, L_X)$ integrates to a twisted presymplectic groupoid $\G \tto X$ and the Lie subalgebroid $B_{S, \gamma}$ to a subgroupoid $\HH \tto S$, then the maps of quotient stacks $[S/\HH] \too [X/\G]$ and $[M/\G] \too [X/\G]$ have Lagrangian structures. Their fiber product is the reduction $\Phi^{-1}(S)/\HH$, which is therefore 0-shifted symplectic.

Many familiar reduction procedures fall under the scope of our main theorem. Among them are Marsden--Weinstein reduction at a point and along a coadjoint orbit \cite{mar.wei:74}, parabolic and unipotent reduction for the action of a complex reductive group, and the reduction of quasi-Hamiltonian and quasi-Poisson manifolds introduced by Alekseev, Malkin, and Meinrenken \cite{ale.mal.mei:98} and by Alekseev, Kosmann-Schwarzbach, and Meinrenken \cite{ale.kos.mei:02}. When $M$ is Hamiltonian or quasi-Hamiltonian for the action of a compact Lie group, our theorem recovers symplectic implosion as defined by Guillemin, Jeffrey, and Sjamaar \cite{gui.jef.sja:02}, as well as its quasi-Hamiltonian counterpart introduced by Hurtubise, Jeffrey, and Sjamaar \cite{hur.jef.sja:06}. Moreover, in the case when $M$ is a symplectic manifold, our results can be specialized to the procedure of symplectic reduction along a submanifold which was studied by the second named author together with Crooks \cite{cro.may:21}.

In addition, this theorem provides a Dirac perspective on several constructions in Poisson and quasi-Poisson geometry. By taking fiber products along strong Dirac maps, we give a new interpretation of the quasi-Hamiltonian reduction of a fusion product in terms of Dirac reduction along a diagonal. We also show that our reduction procedure satisfies a kind of ``universality'' that already appears in symplectic and quasi-Hamiltonian implosion---any reduction of $M$ can be obtained by first reducing the presymplectic groupoid of $X$ and then taking a fiber product.

Motivated by examples from geometric representation theory, we give several applications to quasi-Poisson geometry that rely on the structure of semisimple and reductive groups. We recover a special case of quasi-Poisson reduction relative to a subgroup, which was developed by Li-Bland and \v{S}evera \cite{lib.sev:15} and by \v{S}evera \cite{sev:15}. Specializing to the case when $G$ is a semisimple complex group, we also introduce a quasi-Poisson analogue of Whittaker reduction, a type of Hamiltonian reduction which was first used by Kostant \cite{kos:79}. Because it can be viewed as a type of unipotent reduction at a nilpotent character, it plays a major role in many areas of representation theory and mathematical physics. We construct its multiplicative counterparts in Theorem \ref{whittaker} using generalizations of the Steinberg cross-section \cite{ste:65} defined by He and Lusztig \cite{lus.he:12} and by Sevostyanov \cite{sev:11}, extending previous results of the first named author in \cite{bal:21}. 

\begin{theorem*}
Fix a Weyl group element $w$, let $U$ be the unipotent subgroup associated to its normal representative $\dot{w}$, and let $Z$ be the centralizer of $\dot{w}$ in the corresponding Levi. Write $\Sigma$ for the slice associated to $w$ and $\Theta$ for its $U$-saturation. If $\Phi:M\too G$ is a quasi-Poisson moment map, then 
\[\Phi^{-1}(\Sigma)\cong\Phi^{-1}(\Theta)/U\]
is a Hamiltonian quasi-Poisson $Z$-manifold.
\end{theorem*} 

We also apply our main theorem to develop quasi-Hamiltonian versions of the open Moore--Tachikawa spaces. These are manifolds which were introduced by Ginzburg and Kazhdan \cite{gin.kaz:21} as a first step towards the construction of the Moore--Tachikawa functor \cite{moo.tac:12}---a conjectural 2-dimensional topological quantum field theory valued in a category of holomorphic symplectic varieties. The multiplicative analogue of this conjecture requires a functor which takes values in a category of quasi-Hamiltonian varieties. We construct smooth open submanifolds of these varieties, and we show in Theorem \ref{mooretac} that they satisfy the gluing laws predicted by the Moore--Tachikawa conjecture.

\begin{theorem*}
Let $\Sigma_\Delta$ be the diagonal Steinberg cross-section of the $n$-fold product $G^n$, let $\mathcal{Z}$ be the multiplicative universal centralizer of $G$, and let $\HH$ be the kernel of the $n$-fold multiplication map
\[\mathcal{Z}\times_\Sigma\ldots\times_\Sigma\mathcal{Z}\too\mathcal{Z}.\]
Then $\Sigma_\Delta$ is a reduction level of $G^n$ with stabilizer subgroupoid $\HH$, and the reduction $\mathfrak{Z}_n$ of the fusion double $G^n\times \Gad^n$ at $\Sigma_\Delta$ is a quasi-Hamiltonian $G^n$-manifold. There is a natural isomorphism of quasi-Hamiltonian manifolds
\[\mathfrak{Z}_{m+n-2}\cong(\mathfrak{Z}_m\circledast\mathfrak{Z}_n)\sll{1}G.\]
\end{theorem*} 

\subsection*{Outline}
In Section \ref{first} we recall some background on Dirac geometry, including strong Dirac maps and quasi-Poisson spaces. We also review some technical conditions under which a Dirac structure can be pushed forward along a quotient map. 

In Section \ref{second} we introduce the notion of a reduction level, and we prove our main theorem on Dirac reduction. We end this section by giving many familiar examples of reduction in Poisson, quasi-Poisson, and Dirac geometry that fit into the framework of the main theorem.

In Section \ref{third} we show that products of strong Dirac maps can be reduced along the diagonal. We use this perspective to give Dirac interpretations of several reduction procedures in Poisson and quasi-Poisson geometry, including the quasi-Hamiltonian reduction of a fusion product. 

In Section \ref{qpsec} we apply our results to the quasi-Poisson framework to define multiplicative analogues of a number of constructions in geometric representation theory. We introduce multiplicative versions of Whittaker reduction and of the open Moore--Tachikawa varieties.

In Section \ref{fifth} we give a global counterpart of our main result in the case when the strong Dirac map is a presymplectic realization. This extends our main reduction theorem to quotients by groupoids which are not necessarily source-connected. We also explain how these results can be interpreted in terms of shifted symplectic geometry.

\subsection*{Acknowledgments}
We would like to thank Pavol \v{S}evera for making us aware of his work with David Li-Bland on quasi-Poisson reduction, and Henrique Bursztyn, Alejandro Cabrera, Rui Fernandes, Eckhard Meinrenken, and Marco Zambon for their interesting and constructive comments. We are also grateful to the anonymous referee for many thoughtful suggestions that improved this paper. During the completion of this work, A.B. was partially supported by a National Science Foundation MSPRF under award number DMS--1902921, and M.M. by a Discovery Grant (RGPIN-2023-04587) of the Natural Sciences and Engineering Research Council of Canada (NSERC).

%
%
%
%
%
%
%
%\newpage
\section{Dirac geometry}
\label{first}
We begin by establishing some conventions on Dirac manifolds, and we refer to \cite{bur:13, cra.fer.mar:21, gua:11, mei:18, sev.wei:01} for more details. In Section \ref{1.2} we review the notion of strong Dirac map, and in Section \ref{qpintro} we recall the particular example of quasi-Poisson manifolds. In Section \ref{1.3} we give some technical conditions under which Dirac structures can be pushed forward along surjective submersions.

%
%
%
%
%
%
%
%\newpage
\subsection{Strong Dirac maps}
\label{1.2}
Let $(M,L_M)$ and $(N, L_N)$ be (real or complex) Dirac manifolds twisted by closed $3$-forms $\eta_M$ and $\eta_N$ respectively. A smooth map $f:M\longrightarrow N$ is a \emph{strong Dirac map} if it is forward-Dirac and satisfies 
\begin{equation}
\label{strong}
\Phi^*\eta_N=\eta_M\quad\text{and}\quad \ker \Phi_*\cap \ker L_M=0.
\end{equation}
Such maps were first studied in \cite{bur.cra:05}, where they were called \emph{Dirac realizations}, and subsequently in \cite{ale.bur.mei:09}, where the current name was introduced.

In particular, $\Phi$ is a strong Dirac map if and only if it induces an isomorphism between $\ker L_M$ and $\ker L_N$. This condition implies that, for any $(v,\alpha)\in L_N$, there exists a unique $w\in TM$ such that
\[\Phi_*w=v\quad\text{and}\quad (w,\Phi^*\alpha)\in L_M.\]
Therefore, we obtain a Lie algebroid action \cite[Corollary 3.12]{bur.cra:05} of $L_N$ on $M$ via
\begin{align*}
\rho_M:L_N\times_NM&\longrightarrow TM\\
		(v,\alpha)\,\,\,&\longmapsto \,\,w.
\end{align*}

\begin{example}
\cite[Example 2.4]{bur.cra:09} Let $N$ be a Poisson manifold and let $\Phi:M\longrightarrow N$ be a strong Dirac map. Because $\Phi$ intertwines the $3$-form twists and induces an isomorphism on Dirac kernels, $M$ is also a Poisson manifold and therefore $\Phi$ is a Poisson map. It follows that strong Dirac maps to $N$ are in one-to-one correspondence with Poisson maps to $N$. 

In particular, let $N=\mathfrak{g}^*$ be the dual of a Lie algebra, equipped with the canonical Kirillov--Kostant--Souriau Poisson structure. Any Poisson map $\Phi:(M,\pi_M)\longrightarrow (\mathfrak{g}^*,\pi_{\mathrm{KKS}})$ is a strong Dirac map, so it induces a $\mathfrak{g}$-action
\begin{align*}
\rho_M:\g&\longrightarrow TM\\
		\xi&\longmapsto \pi_M^\#( \Phi^*(\xi)).
		\end{align*}
With respect to this action $\Phi$ is a moment map---in other words, strong Dirac maps to $(\mathfrak{g}^*,\pi_{\mathrm{KKS}})$ are in one-to-one correspondence with Hamiltonian Poisson $\mathfrak{g}$-spaces.
\end{example}

\begin{example}
\cite[Theorem 3.16]{bur.cra:05}
If $G$ is a  Lie group whose Lie algebra $\g$ carries an invariant, nondegenerate, symmetric bilinear form, we can consider quasi-Poisson structures with respect to the action of $G$, as we recall in the next section. In this setting, any quasi-Poisson moment map is a strong Dirac map \cite[Proposition 3.19]{bur.cra:05}. Conversely, any strong Dirac map from $(M,L_M)$ to $(G,\pi_G)$ induces on $M$ an action of $\mathfrak{g}$ with respect to which the Dirac structure $L_M$ is quasi-Poisson, and whose associated moment map is $\Phi$ \cite[Proposition 3.20]{bur.cra:05}. Therefore, strong Dirac maps to $(G, L_G)$ are in one-to-one correspondence with Hamiltonian quasi-Poisson $\mathfrak{g}$-spaces.
\end{example}

%
%
%
%
%
%
%
%\newpage
\subsection{Quasi-Poisson manifolds}
\label{qpintro}
In this section we review Dirac structures associated to quasi-Poisson bivectors. Let $\mathfrak{g}$ be a Lie algebra equipped with an invariant, nondegenerate, symmetric bilinear form, let $G$ be a Lie group that integrates it, and let $\eta_G$ be the bi-invariant Cartan $3$-form on $G$. Under the isomorphism between $\g$ and $\g^*$ given by the inner product, the Cartan $3$-form corresponds to a $3$-tensor $\chi\in\wedge^3\mathfrak{g}.$

If $\g$ acts on a manifold $M$ via
\begin{align*}
\g&\longrightarrow TM \\
\xi&\longmapsto \xi_M,
\end{align*}
a \emph{quasi-Poisson} structure on $M$ is a $\g$-invariant bivector $\pi$ with the property that
\[[\pi,\pi]=\chi_M.\]
This bivector is \emph{Hamiltonian} if there exists a group-valued moment map
\[\Phi:M\longrightarrow G,\]
equivariant relative to the conjugation action of $G$ on itself, which satisfies a multiplicative analogue of the usual moment map condition \cite[Definition 2.2]{ale.kos.mei:02}. 

By \cite[Theorem 3.16]{bur.cra:05}, the Hamiltonian quasi-Poisson structure $\pi$ corresponds to the $-\Phi^*\eta_G$-twisted Dirac structure
\begin{equation}
\label{qpdirac}
L=\{(\pi^\#(\alpha)+\xi_M,C^*(\alpha)+\Phi^*\sigma(\xi))\mid\alpha\in T^*M, \xi\in\g\}.
\end{equation}
Here 
\begin{align*}
\sigma:\g&\too T^*G\\
		\xi&\longmapsto \tfrac{1}{2}(\xi^L+\xi^R)^\vee,
		\end{align*}
where $x^\vee$ is the dual of $x\in TG$ under the identification $TG\cong T^*G$ given by the nondegenerate bilinear form, and 
\[C^*=\Id-\frac{1}{4}e_{iM}\otimes\Phi^*(\theta_i^L-\theta_i^R),\]
where $\{e_i\}$ is an orthonormal basis of $\g$ and $\theta_i^L$ and $\theta_i^R$ are the corresponding left- and right-invariant $1$-forms on $G$. In particular, the projection of $L$ onto $TM$ is the image of the bundle map
\begin{align*}
T^*M\oplus\mathfrak{g}&\longrightarrow TM \\
		(\alpha,\xi)&\longmapsto\pi^\#(\alpha)+\xi_M.
		\end{align*}
The foliation integrating this distribution is the foliation of $M$ by \emph{quasi-Hamiltonian} leaves.

\begin{example}
\label{cd}
The group $G$ has a natural quasi-Poisson structure relative to the conjugation action \cite[Proposition 3.1]{ale.kos.mei:02}, which is called the \emph{Cartan--Dirac structure}. It can be viewed as a multiplicative analogue of the classical Kirillov--Kostant--Souriau Poisson structure on $\g^*$. Its moment map is the identity, and the associated Dirac structure is
\begin{equation}
\label{cdexp}
L_G\coloneqq\left\{\left(\xi^L-\xi^R,\sigma(\xi)\right)\mid\xi\in\g\right\}.
\end{equation}
Projecting onto the tangent bundle, the nondegenerate leaves of the Cartan--Dirac structure are the conjugacy classes of $G$. In particular,  there is a natural isomorphism 
\[L_G\cong G\rtimes\g\]
between the Cartan--Dirac structure and the action Lie algebroid on $G$ corresponding to the conjugation action.
\end{example}

\begin{example}
\label{fusdouble}
The quasi-Hamiltonian \emph{internal fusion double} $D(G)\coloneqq G\times G$ \cite[Example 5.3]{ale.kos.mei:02} is a nondegenerate quasi-Poisson $G\times G$-manifold with respect to the action
\[(g,h)\cdot (a,b)=(gah^{-1},hbh^{-1}).\]
The associated moment map is given by 
\begin{align*}
\Phi:D(G)&\longrightarrow G\times G\\
	(a,b)&\longmapsto (aba^{-1},b^{-1}),
	\end{align*}
and the quasi-Poisson structure on $D(G)$ is a multiplicative counterpart of the canonical symplectic structure on the cotangent bundle of $T^*G$. Viewed as a groupoid over $G$ with source and target maps given by the two components of $\Phi$, the double $D(G)$ is a presymplectic groupoid integrating the Cartan--Dirac structure $L_G$.
\end{example}

%
%
%
%
%
%
%
%\newpage
\subsection{Quotients of Dirac structures}
\label{1.3}
Let $(M, L_M)$ be an $\eta_M$-twisted Dirac manifold. A local vector field $Z$ on $M$ is \emph{Dirac} if its flow consists of local Dirac diffeomorphisms. This occurs if and only if the Lie derivative $\L_Z$ preserves the sheaf of sections of $L_M$ and the $3$-form $\L_Z\eta_M$ vanishes along the presymplectic leaves of $M$. A distribution $D \subseteq TM$ is called a \emph{Dirac distribution} if every vector in $D$ extends to a Dirac local section of $D$, and a Dirac distribution $D$ is called \emph{regular} if $D \cap L_M$ has constant rank.

\begin{remark}
\label{mzq8z1k7}
Suppose that $D \s TM$ is a distribution contained in $L_M \cap \ker \eta_M$. If $Z$ is a local section of $D$, then 
\[ \llbracket (Z, 0), (v, \alpha)\rrbracket=(\L_Zv, \L_Z\alpha) \qquad\text{for all }(v,\alpha)\in L_M.\]
Since $L_M$ is closed under the Dorfman bracket, $Z$ is a Dirac vector field and therefore $D$ is a regular Dirac distribution.
\end{remark}

Given an involutive distribution $D$ on $M$, we will say that its leaf space $M/D$ ``is a manifold'' if it has the structure of a smooth manifold such that the quotient map is a smooth submersion. The next proposition shows that in this case, if $D$ is a regular Dirac distribution and compatible with $\eta_M$, then the Dirac structure on $M$ pushes forward to a Dirac structure on $M/D$.

\begin{proposition}
\label{bgfkcgj1}
Let $D \subseteq TM$ be an involutive regular Dirac distribution contained in $\ker \eta_M$, and suppose that the leaf space $Q \coloneqq M/D$ is a manifold with quotient map $\pi : M \longrightarrow Q.$
\begin{enumerate}[label=\textup{(\arabic*)}]
\item\label{tpcp6ia5}
The pushforward 
\[L_Q \coloneqq \pi_*L_M\]
is an $\eta_Q$-twisted Dirac structure on $Q$, where $\eta_Q$ is uniquely characterized by the property that $\pi^*\eta_Q = \eta_M$.
\item\label{15i0o7u8}
If $(Z,L_Z)$ is an $\eta_Z$-twisted Dirac manifold and $\varphi : (M,L_M) \too (Z, L_Z)$ is an f-Dirac map such that $D$ is contained in the kernel of $\varphi_*$, then $\varphi$ descends to an f-Dirac map 
\[\phibar : (Q,L_Q) \too (Z,L_Z).\]
Moreover, if $D$ contains $\ker \varphi_* \cap L_M$ and if $\varphi^*\eta_Z=\eta_M$, then $\phibar$ is a strong Dirac map.
\item\label{gpxy5jhi}
If $D$ is contained in the kernel of $L_M$, then the f-Dirac map 
\[\pi : (M,L_M) \too (Q, L_Q)\]
is also b-Dirac.
\item\label{r0dvmlu4}
If $D$ is contained in the kernel of $L_M$, the nondegenerate leaves of $Q$ are the images of nondegenerate leaves of $M$. In particular, if $M$ is nondegenerate and $L_M = L_\omega$ is the graph of a $2$-form $\omega$, then $L_Q$ is also nondegenerate, and the corresponding 2-form $\overline{\omega}$ is uniquely characterized by the property that $\pi^*\overline{\omega} = \omega$. 
\end{enumerate}
\end{proposition}

\begin{proof}
(1) Since $D$ is contained in $\ker \eta_M$ and $d\eta_M = 0$, there is a unique closed 3-form $\eta_Q$ on $Q$ such that $\pi^*\eta_Q = \eta_M$. By \cite[Proposition 1.13]{bur:13}, it then suffices to show that $L_M$ is $\pi$-invariant. Let $p$ and $q$ be points of $M$ such that $\pi(p)=\pi(q)$, and suppose that $Z$ is a Dirac local section of $D$ whose flow $f$ at some time satisfies $f(p)=q.$ Since the local automorphism $(f_*, f^{-1*})$ of $\TM$ preserves $L_M$, for any $(\pi_*v,\alpha)\in\pi_*L_{M,p}$ we have
\[(f_*v, f^{-1*}\pi^*\alpha)=(f_*v,\pi^*\alpha) \in L_{M, q}.\]
It follows that
\[(\pi_*v, \alpha)=(\pi_*f_*v,\alpha)\in \pi_*L_{M,q},\]
so $\pi_*L_{M,p}$ is contained in $\pi_*L_{M,q}$. Since these are vector spaces of the same dimension, they are equal and therefore $L_M$ is $\pi$-invariant.

(2) Suppose that $D$ is contained in $\ker \varphi_*$ and let $\phibar : Q \too Z$ be the induced map. Then 
\[\phibar_*L_Q = \phibar_* \pi_* L_M = \varphi_* L_M = L_Z,\] 
so $\phibar$ is f-Dirac. If $D$ contains $\ker \varphi_* \cap L_M$, then
\[\ker\phibar_*\cap L_Q=\ker\phibar_*\cap\pi_* L_M=\pi_*(\ker \varphi_*\cap L_M)=0.\]
Moreover, 
\[\pi^*\phibar^*\eta_Z=\varphi^*\eta_Z=\eta_M\]
and so, by uniqueness, $\phibar^*\eta_Z=\eta_Q$. Therefore $\phibar$ is a strong Dirac map.

(3) Suppose that $D$ is contained in the kernel of $L_M$ and let $(u,\pi^*\alpha)\in\pi^*\pi_* L_M,$ so that $(\pi_*u, \alpha)\in \pi_*L_M$. Then there is an element $v$ in $TM$ such that $\pi_*v=\pi_*u$ and $(v,\pi^*\alpha)\in L_M$, so $u - v$ is contained in $D$ and therefore also in the kernel of $L_M$. It follows that
\[(u, \pi^*\alpha) = (v, \pi^*\alpha) + (u - v, 0) \in L_M.\]
Since both sides are distributions on $M$ of the same rank, they are equal and so $\pi$ is b-Dirac.

(4) Suppose once again that $D$ is contained in the kernel of $L_M$. Then
\[p_T(\pi_*L_M)=\pi_*p_T(L_M\cap (TM\oplus\im\pi^*))=\pi_*p_T(L_M),\]
where the second equality follows from the fact that $\pi$ is b-Dirac by part \ref{gpxy5jhi}. In particular, if $L_M$ is nondegenerate and $L_M=L_\omega$, this implies that $L_Q$ is also nondegenerate. In this case, whenever $Z$ is a Dirac vector field we have 
\[\L_Z\omega = 0\quad\text{and}\quad \imath_Z\omega = 0,\]
where the first equality follows by definition and the second follows from the fact that $D$ is in the kernel of $L_M$. Therefore there is a unique 2-form $\overline{\omega}$ on $Q$ such that $\pi^*\overline{\omega} = \omega$, and so $L_Q=L_{\overline{\omega}}$.
\end{proof}

\begin{remark}
\label{strata1}
More generally, $M=\sqcup M^i$ may be a union of manifolds in which each stratum $M^i$ is a Dirac manifold. Then we can apply Proposition \ref{bgfkcgj1} to any collection $D=\{D^i\subset TM^i\}$ of involutive regular Dirac distributions for which the leaf spaces $M^i/D^i$ are smooth manifolds. We then obtain a topological quotient
\[M/D\coloneqq\bigsqcup M^i/D^i\]
of $M$ which is a disjoint union of Dirac manifolds.
\end{remark}

\vspace{.1in}

%
%
%
%
%
%
%
%\newpage
\section{Dirac reduction}
\label{second}
We are now ready to introduce reduction levels and relative reduction levels of Dirac manifolds, in Section \ref{2.1}. In Section \ref{2.2} we give a large number of examples, and in Section \ref{2.3} we prove our main reduction Theorem \ref{main}. In Section \ref{2.4} we illustrate this theorem by showing that it recovers a variety of familiar constructions from symplectic geometry and representation theory.  

\subsection{Reduction levels in Dirac manifolds}
\label{2.1}
{Let $X$ be a manifold equipped with a closed $3$-form $\eta_X$, and let $L_X$ be an $\eta_X$-twisted Dirac structure on $X$. Recall that a generalized submanifold of $X$ \cite{gua:11} is a pair $(S,\gamma)$ consisting of a submanifold $\imath:S\hooklongrightarrow X$ and a 2-form $\gamma\in\Omega^2(S)$ such that
\[\imath^*\eta_X+d\gamma=0.\]
The \emph{generalized tangent bundle} of $(S,\gamma)$ is the subbundle of $\TM$ given by the Dirac pushforward of the graph of $\gamma$---in other words,
\[\T^\gamma S\coloneqq \{(v,\alpha)\in \TM\mid v\in TS\,\text{ and }\,\imath^*\alpha=\gamma^\flat(v)\}.\]}

\begin{definition}
\label{levels}
{A \emph{reduction level} of $X$ is a generalized submanifold $(S,\gamma)$ such that the distribution
\[A_{S,\gamma} \coloneqq L_X \cap \T^\gamma S\]
has constant rank.
Proposition \ref{subalg} shows that in this case $A_{S,\gamma}$ is a Lie subalgebroid of $L_X$.} 
\end{definition}

In the special case when $\gamma= 0$ we simply write $S$ for the reduction level $(S,0)$ and $A_{S}$ for the corresponding stabilizer subalgebroid. Definition \ref{levels} has the following generalization, which we will use throughout the remainder of the paper.

\begin{definition}
\label{genlevels}
{Let $(Z, L_Z)$ be an $\eta_Z$-twisted Dirac manifold and let $\overline{Z}$ denote the same manifold with the opposite Dirac structure
\[-L_Z\coloneqq\{(v,-\alpha)\mid (v,\alpha)\in L_Z\}.\]
A \emph{$Z$-level} or \emph{relative reduction level} of $X$ is a reduction level $(S,\gamma)$ of $X\times\overline{Z}$ such that the natural projections}
\begin{equation*}
\begin{tikzcd}
&S\arrow[ld, swap, "p_X"]\arrow[rd, "p_Z"]&\\
X&&\overline{Z},									&
\end{tikzcd}
\end{equation*}
{satisfy
\begin{equation}
\label{relativebrane}
p_{Z*}(\tau_{{-\gamma}} p_X^*L_X) = L_Z,
\end{equation}
where $\tau_{{-\gamma}}$ is the gauge transformation by {$-\gamma$}. Proposition \ref{subalg} shows that in this case the kernel
\[B_{S,\gamma}\coloneqq \ker\left[A_{S,\gamma}\too L_{\overline{Z}}\right]\]
of the pointwise projection from $A_{S,\gamma}$ to $L_{\overline{Z}}$ is a Lie algebroid, which is naturally identified with a subalgebroid of $L_X$ and which we call the \emph{stabilizer subalgebroid of} $(S, \gamma)$.}
\end{definition}

Note that, while the Lagrangian distribution $p_X^*L_X$ is not necessarily a smooth bundle and therefore may fail to be a Dirac structure on $S$, in the first condition of Definition \ref{genlevels} we can nevertheless define its gauge transform by {$-\gamma$} and its pushforward along $p_Z$ pointwise.

When $Z$ is a single point, Definition \ref{genlevels} specializes to Definition \ref{levels} and we use the same notation. In particular, in this case we have $B_{S,\gamma}=A_{S,\gamma}$, so for consistency we will use the notation $B_{S,\gamma}$ throughout. Moreover, as in the previous case, when $\gamma=0$ we write $S$ for the $Z$-level $(S,0)$ and $B_{S}$ for the stabilizer subalgebroid $B_{S,0}$. 

\begin{proposition}
\label{subalg}
Let $(S,\gamma)$ be a $Z$-level of $X$. Then the distribution $B_{S,\gamma}$ is a Lie subalgebroid of $L_{X\times\overline{Z}}$.
\end{proposition}

\begin{proof}
{First we show that the pointwise projection $A_{S,\gamma}\too L_{\overline{Z}}$ is surjective. Suppose that $(v,\alpha)$ is a point of $L_{\overline{Z}}$. By condition \eqref{relativebrane}, there is a vector $w\in TS$ with $p_{Z*}w=v$ and a covector $\beta\in T^*X$ satisfying
\[(w,-p_Z^*\alpha+\gamma^\flat(w))=(w, p_X^*\beta)\in p_X^*L_X.\]
Then the pair $(w,p_X^*\beta+p_Z^*\alpha)$ belongs to $A_{S,\gamma}$, and its projection onto $L_{\overline{Z}}$ is precisely $(v,\alpha).$}

{This surjectivity, together with the constant-rank condition on $B_{S,\gamma}$, implies that the kernel $B_{S,\gamma}$ has constant rank, and is therefore a smooth subbundle of $L_{X\times\overline{Z}}$. It remains to show that for any two sections $(v, \alpha)$ and $(w, \beta)$ of $L_{X\times\overline{Z}}$ which restrict to sections of $B_{S,\gamma}$, the bracket 
\[\llbracket(v, \alpha), (w, \beta) \rrbracket\]
also restricts to a section of $B_{S,\gamma}$. Since $v$ and $w$ are contained in the kernel of $p_{Z*}$, so is $[v, w]$. Moreover,
\begin{align*}
(\imath_{v\wedge w}(\eta_X-\eta_Z))_{\vert S} = -\imath_{v\wedge w} d\gamma 
				&= \imath_w d(\imath_v \gamma)-\L_v (\imath_w \gamma) + \imath_{[v, w]}\gamma \\
				&=  \imath_wd(\imath^*\alpha)-\L_v(\imath^*\beta)+  \imath_{[v, w]}\gamma,
\end{align*}
so
\[\gamma^\flat[v, w]= \L_v(\imath^*\beta)-\imath_wd\alpha_{\vert S}+\imath_{v\wedge w}(\eta_X-\eta_Z)_{\vert S}\]
and therefore $\llbracket(u, \alpha), (v, \beta) \rrbracket$ restricts to a section of $B_{S,\gamma}$.}
\end{proof}

\begin{remark}
\label{oldlevels}
{A special class of relative reduction levels arises when $S$ is a submanifold of $X$, $\gamma$ is a $2$-form on $S$, and $\varphi:S\too Z$ is a smooth map of constant rank to a Dirac manifold $(Z,L_Z)$. The graph of $\varphi$ is a $Z$-level of $X$ if and only if 
\begin{itemize}[topsep=1.5pt]
\item $\varphi_*(\tau_{-\gamma}\imath^*L_X) = L_Z$ and $\imath^*\eta_X + d\gamma=\varphi^*\eta_Z$, and
\item the distribution
\begin{align*}
B &\coloneqq L_X \cap \imath_* L_\gamma\cap (\ker \varphi_* \oplus T^*X)\\
					&\,\,= \{(v, \alpha) \in L_X \mid v \in \ker \varphi_* \text{ and } \imath^*\alpha=\gamma^\flat(v)\}
					\end{align*}
has constant rank,
\end{itemize}
\noindent and in this case $B$ is its stabilizer subalgebroid.}
\end{remark}

%
%
%
%
%
%
%
%\newpage
\subsection{Examples of reduction levels}
\label{2.2}
Many types of submanifolds in Poisson, quasi-Poisson, and Dirac geometry can be viewed as reduction levels. We begin with a series of familiar examples.

\begin{example}
\label{ex:poisson2}
If $\gamma= 0$, then $S$ is a reduction level of $X$ if and only if
\[B_S=L_X\cap (TS\oplus TS^\circ)\]
has constant rank. When $X$ is a Poisson manifold, this implies that a submanifold of $X$ is a reduction level if and only if it is a pre-Poisson submanifold \cite{cat.zam:09}. In particular, any coisotropic submanifold of $X$ whose characteristic distribution has constant rank is a reduction level.
\end{example}

\begin{example}
\label{point}
Any point $x\in X$ is a reduction level, and its stabilizer subalgebroid
\[B_x=L_{X,x}\]
is the isotropy Lie algebra of $L_X$ at $x$.
\end{example}

\begin{example}
Suppose that the Dirac manifold $(X,L_X)$ is untwisted, and that $\O$ is a nondegenerate leaf of $X$. Then
\[B_\O=L_X\cap (T\O\oplus T\O^\circ)=(\ker L_X \oplus (L_X \cap T^*X))|_\O,\]
and $\O$ is a reduction level of $X$ when this distribution has constant rank. In particular, in the case when $X=\g^*$ is the dual of a Lie algebra, any coadjoint orbit $\O$ is a reduction level with stabilizer subalgebroid
\[B_{\O}=\{(\xi,x)\in\mathcal{O}\times\g\mid \ad_x^*(\xi)=0\}.\]
\end{example}

\begin{example}
\label{orbit}
Let $\O$ be a nondegenerate leaf of $(X,L_X)$ with presymplectic form $\omega$. Then $d\omega=-\imath^*\eta_X$ and
\[B_{\O,\omega}=L_X\cap\imath_*L_{\omega}=L_X\cap\imath_*\imath^*L_X=L_{X\vert\O},\]
where the last equality follows since the inclusion map of a presymplectic leaf is both b-Dirac and f-Dirac \cite[Example 3.5]{bur.cra:05}. This distribution always has constant rank, and therefore $(\O,\omega)$ is always a reduction level of $X$. In the special case when $X=\g^*$ is the dual of a Lie algebra and $\O$ is a coadjoint orbit with symplectic form $\omega$, the stabilizer subalgebroid is the action Lie algebroid
\[B_{\O,\omega}=\mathcal{O}\rtimes\g.\]
\end{example}

\begin{lemma}
\label{ptrans}
Suppose that $\imath:S\hooklongrightarrow X$ is a submanifold such that 
\begin{itemize}
\item $S$ intersects every nondegenerate leaf of $X$ cleanly,
\item $\imath^*\eta_{X}=0$, and
\item the kernel of $\imath^*L_X$ has constant rank.
\end{itemize}
Then $S$ is a reduction level of $X$ with stabilizer subalgebroid $B_S=\ker\imath^*L_X$. 
\end{lemma}
\begin{proof}
As in the proof of \cite[Lemma 1.18]{bal:21}, cleanness implies that 
\[L_X\cap TS^\circ=0,\]
and in particular that $\imath^*L_X$ is a non-twisted Dirac structure on $S$. Moreover, we have
\[\ker\imath^* L_X\cong\frac{L_X\cap(TS\oplus TS^\circ)}{L_X\cap TS^\circ}\cong L_X\cap(TS\oplus TS^\circ).\]
In view of Example \ref{ex:poisson2}, it follows that $B_S=\ker\imath^*L_X$ has constant rank, and therefore that $S$ is a reduction level of $(X,L_X)$. 
\end{proof}

\begin{example}
A submanifold $\imath:S\hooklongrightarrow X$ is called a \emph{clean Poisson--Dirac submanifold} \cite{bra.fre.mar:20} if $S$ intersects every nondegenerate leaf $\O$ of $X$ cleanly, and the restriction of the presymplectic form $\omega_{\O}$ to the intersection $S\cap\O$ is a symplectic form. In this case, the pullback $\imath^*L_X$ is a Poisson structure, so its kernel is trivial. In view of Lemma \ref{ptrans}, this means that $S$ is a reduction level of $X$ with trivial stabilizer subalgebroid. In particular, any Poisson transversal is a reduction level with trivial stabilizer subalgebroid.
\end{example}

Reduction levels and relative reduction levels also appear in a number of examples from geometric representation theory. We recall some of these examples here. Let $\g$ be a reductive complex Lie algebra, and equip it with a nondenerate invariant symmetric bilinear form. The resulting isomorphism $\g\cong\g^*$ gives the Lie algebra $\g$ a Kirillov--Kostant--Souriau Poisson structure. Let $\p$ be a parabolic subalgebra with Levi decomposition $\p=\l +\uu$.

\begin{example} 
\label{unips}
The distribution
\begin{align*}
B_{\uu}&= L_{\g}\cap (T\uu\oplus T\uu^\circ)\\
			&=\{(\xi,x)\in\uu\times\g\mid x\in\p\}\\
			&=\uu\rtimes\p
			\end{align*}
has constant rank, so $\uu$ is a reduction level of $\g$ with stabilizer subalgebroid $\uu\rtimes\p$.
\end{example}

\begin{example}
\label{parabs}
The nondegenerate form on $\g$ restricts to a nondegenerate form on the reductive subalgebra $\l$, which allows us to identify $\l\cong\l^*$. The quotient map 
\[\varphi:\p\too\l\]
fits into the commutative diagram
\begin{equation*}
\begin{tikzcd}[column sep=large, row sep=large]
\p			\arrow[r, hook, "\imath"] \arrow[d, swap, "\varphi"]			&\g\arrow[ld, "f"]		\\
\l,									&
\end{tikzcd}
\end{equation*}
where the horizontal arrow is the inclusion and the diagonal map $f$ is induced by the restriction $\g^*\too\l^*$ and is therefore a Poisson map. Identifying $L_{\g}$ with the action Lie algebroid $\g\rtimes\g$, the kernel of the distribution $\imath^*L_{\g}$ is precisely $\p\rtimes\uu,$ and therefore $\varphi_*\imath^*L_{\g}=f_*L_{\g}=L_{\l}$. 

Viewed as a submanifold of $\g\times\l$ via the embedding
\[(\imath, \varphi):\p\hooklongrightarrow\g\times\l\]
as in Remark \ref{oldlevels}, the subspace $\p$ is an $\l$-level of $\g$ with stabilizer subalgebroid
\begin{align*}
B_{\p}&= L_{\g}\cap (T\p\oplus T\p^\circ)\cap (\ker\varphi_*\oplus T^*\g)\\
			&=\{(x,\xi)\in\p\times\g\mid \xi\in\p\cap\uu\}\\
			&=\p\rtimes\uu.
			\end{align*}
\end{example}

The last two examples have counterparts when $G$ is a complex group integrating $\g$, equipped with the Cartan--Dirac quasi-Poisson structure, and $P$ is the parabolic subgroup with Lie algebra $\p$, unipotent radical $U$, and Levi quotient $L\coloneqq P/U$. The corresponding reduction levels are special cases of a more general class of reduction levels of Cartan--Dirac structures which we discuss in Section \ref{qpsec}.

\begin{example}
Since the nilradical $\uu$ satisfies $(\uu,[\uu,\uu])=0$, the Cartan $3$-form of $G$ vanishes along $U$. Using the identification \eqref{cdexp}, 
\begin{align*}
B_{U}&=L_G\cap (TU\oplus TU^\circ)\\
			&=\{(\xi^R-\xi^L,\sigma(\xi))\mid \xi\in\p\}\\
			&=U\rtimes\p
			\end{align*}			
is a distribution of constant rank, so $U$ is a reduction level of $G$ with stabilizer subalgebroid $U\rtimes\p$.
\end{example}

\begin{example}
Let $\imath:P\hooklongrightarrow G$ be the inclusion of a parabolic subgroup with unipotent radical $U$, and write 
\[\varphi:P\longrightarrow L\]
for the projection of $P$ onto the Levi. Then
\begin{align*}
\varphi_*\imath^*L_G&=\varphi_*\{(\xi^R-\xi^L, \sigma(\xi))\mid \xi\in\p\}\\
				&=\{(\xi^R-\xi^L, \sigma(\xi))\mid x\in\l\}\\
				&=L_L.
				\end{align*}
Therefore $P$, viewed as a submanifold of $G\times L$ via
\[(\imath,\varphi):P\hooklongrightarrow G\times L,\]
satisfies condition \eqref{relativebrane}. Moreover, since the Cartan $3$-form vanishes along $U$ we obtain $\varphi^*\eta_L=\imath^*\eta_G$, and we compute the associated distribution as
\begin{align*}
B_{P}&=L_G\cap (\ker\varphi_*\oplus TP^\circ)\\
			&=\{(\xi^R-\xi^L,\sigma(\xi))\mid \xi\in\uu\}\\
			&=P\rtimes\uu.
			\end{align*}			
Since it has constant rank, the submanifold $P$ is an $L$-level of $G$ with stabilizer subalgebroid $P\rtimes\uu$. 
\end{example}

Now let $\g$ be the Lie algebra of a real compact Lie group, choose a maximal Cartan subalgebra $\t$ and a set of simple roots $\Delta$, and denote by $\t^*_+\subset\t^*$ the corresponding closed dominant Weyl chamber. This chamber is a disjoint union
\[\t^*_+=\bigcup_{J\subset\Delta}\sigma_J\]
of faces defined by
\[\sigma_J\coloneqq\{s\in\t^*_+\mid s(\alpha^\vee)=0\text{ if and only if }\alpha\in J\},\]
where we write $\alpha^\vee$ for the coroot corresponding to the simple root $\alpha$. Note that any two elements of $\sigma_J$ have the same centralizer under the coadjoint action, and that this is a connected subgroup of $G$, which we denote by $G_J$. 

\begin{example}
\label{faces}
Writing $\imath:\sigma_J\hooklongrightarrow\g^*$ and $\varphi:\sigma_J\hooklongrightarrow\t^*$ for the inclusion maps, we see that
\[\varphi_*\imath^*L_{\g^*}=L_{\t^*}\]
are both the zero Poisson structure on $\t^*$, since $\sigma_J$ is a cross-section for the coadjoint orbits that it meets. Moreover, viewing $\sigma_J$ as a submanifold of $\g^*\times\t^*$ via the two inclusion maps, we obtain
\begin{align*}
B_{\sigma_J}&= L_{\g^*}\cap (T\sigma_J\oplus T\sigma_J^\circ)\\
			&=\{(s, \xi)\in\sigma_J\times\g\mid \ad_\xi^*s=0, \xi\in T\sigma_J^\circ\}\\
			&=\{(s, \xi)\in\sigma_J\times\g\mid \xi\in[\g_s,\g_s]\}\\
			&=\sigma_J\rtimes[\g_J,\g_J],
			\end{align*}
where we denote by $\g_J$ the Lie algebra of $G_J$ and where the third equality follows, for instance, from the proof of \cite[Lemma 4.27]{cro.may:21}. So $\sigma_J$ is a $\t^*$-level of $\g^*$ with stabilizer subalgebroid $\sigma_J\rtimes [\g_J,\g_J]$, and we can view the closed dominant Weyl chamber $(\t^*_+,\varphi)$ as a stratified $\t^*$-level of $\g^*$.
\end{example}

For a multiplicative counterpart to the previous example, let $G$ be the simply-connected compact Lie group integrating $\g$, equipped with the Cartan--Dirac structure. For each subset $J\subset\Delta$ define
\[S_J\coloneqq\{t\in T\mid \alpha(t)=1\text{ if and only if }\alpha\in J\}\]
to be the image of the corresponding face of the dominant alcove under the exponential map. The centralizer of any element in the submanifold $S_J$ is precisely the connected subgroup $G_J$.

\begin{example} 
\label{alcove}
Writing once again $\imath:S_J\hooklongrightarrow G$ and $\varphi:S_J\hooklongrightarrow T$ for the inclusion maps, we see that $\varphi_*\imath^*L_{G}$ is the zero Poisson structure on $T$ because $S_J$ intersects each conjugacy class in a unique point. We view $S_J$ as a submanifold of $G\times T$ and compute 
\begin{align*}
B_{S_J}&= L_G\cap (TS_J\oplus TS_J^\circ) \\
			&=\{(t,\xi)\in S_J\times\g\mid \xi^R-\xi^L=0, \sigma(\xi)\in TS_J^\circ\}\\
			&=\{(t,\xi)\in S_J\times\g\mid \xi\in[\g_t,\g_t]\}\\
			&=S_J\rtimes[\g_J,\g_J],
			\end{align*}
where once again the third equality follows from the proof of \cite[Lemma 4.27]{cro.may:21}. Since the Cartan $3$-form $\eta_G$ vanishes along $T$, we conclude that $S_J$ is a $T$-level of $G$ with stabilizer subalgebroid $S_J\rtimes [\g_J,\g_J]$. We can therefore view the image $(\overline{\mathcal{A}},\varphi)$ of the closed dominant alcove as a stratified $T$-level of the compact Lie group $G$.
\end{example}

%
%
%
%
%
%
%\newpage
\subsection{Reducing Dirac structures}
\label{2.3}
Let $(M, L_M)$ be an $\eta_M$-twisted Dirac manifold together with a strong Dirac map $\Phi:M\longrightarrow X.$ Condition \eqref{strong} induces a Lie algebroid action
\[\rho_M:L_{X}\times_{X}M\longrightarrow TM.\]
Let $(S,\gamma)$ be a $Z$-level of $X$ and {suppose that the map $\Phi\times p_X$ intersects the diagonal in $X\times X$ cleanly, so that the fiber product $C \coloneqq M\times_XS$ is a smooth manifold. We then have a Cartesian diagram}
\begin{equation}
\label{diagram}
\begin{tikzcd}[row sep=large, column sep=large]
C			\arrow[r, "p_M"] \arrow[d, swap, "\Phi"]			&M	\arrow[d, "\Phi"]\\
S		\arrow[r, "p_X"]							&X,
\end{tikzcd}
\end{equation}
{and the action of $L_{X}$ on $M$ pulls back to an action
\begin{align*}
\rho_C : B_{S,\gamma} \times_S C &\too TC\\
			(v,\alpha)&\mtoo (w,v)
			\end{align*}
of the stabilizer subalgebroid $B_{S,\gamma}$ on $C$, where $(v,\alpha)\in L_X$ and $w$ is the unique element of $TM$ such that $\Phi_*w=v$ and $(w,\Phi^*\alpha)\in L_M.$}

\begin{proposition}
\label{restriction}
{Suppose that the isotropy algebras of $\rho_C$ have constant dimension along $C$. Then
\begin{equation}
\label{twistpull}
L_C \coloneqq \tau_{-\Phi^*\gamma}p_M^*L_M
\end{equation}
is a $\Phi^*(p_X^*\eta_X + d\gamma)$-twisted Dirac structure on $C$ which satisfies 
\[\Phi_*L_C=\tau_{-\gamma}p_X^*L_X.\]}
\end{proposition}

\begin{proof}
{Since $\rho_M(v,\alpha)=0$ implies that $v=0$, the isotropy subalgebras of $\rho_C$ are given by the intersection
\begin{equation}
\label{eq:dim}
\ker\rho_C=\ker\rho_M\cap \ker p_X^*.
\end{equation}
To show that $L_C$ is a Dirac structure, we must show that $p_M^*L_M$ is a smooth bundle, and for this is it sufficient to check that $L_M\cap \ker p_M^*$ has constant rank \cite[Proposition 1.10]{bur:13}.}

{First note that
\begin{align*}
L_M\cap \ker p_M^*&=\{(0,\beta)\in L_M\mid p_M^*\beta=0\}\\
		&=\{(0,\Phi^*\alpha)\in L_M\mid \Phi^*p_X^*\alpha=0\}.
		\end{align*}
Because $\Phi$ is f-Dirac, $(0,\Phi^*\alpha)\in L_M$ implies that $(0,\alpha)\in L_{X}$. At each point of $C$ we therefore get a natural isomorphism
\[L_M\cap \ker p_M^*\cong \frac{\ker\rho_M\cap \ker p_X^*}{\ker\Phi^*}.\]
Since the fiber product is smooth, the bottom of this quotient has constant dimension. Moreover, in view of \eqref{eq:dim} and the assumption of the proposition, the top also has constant dimension. Therefore the dimension of the left-hand side is constant along $C$, and it follows that $L_C$ is a Dirac structure.}

To check the final two conditions we chase diagram \eqref{diagram}. Since $\Phi$ is a strong Dirac map and $\tau_{-\Phi^*\gamma}$ is a gauge transform, the Dirac structure $L_C$ is twisted by
\[p_M^*\Phi^*\eta_X + d\Phi^*\gamma = \Phi^*(p_X^*\eta_X + d\gamma).\]
Moreover,
\begin{align*}
\Phi_*L_C&=\Phi_*\tau_{-\Phi^*\gamma}p_M^*L_M \\
			&=\tau_{-\gamma}\Phi_*p_M^*L_M \\
			&=\tau_{-\gamma}p_X^*L_X,
			\end{align*}
where the second equality {follows by commuting the gauge transformation with the Dirac pushforward} and the third follows from \cite[Lemma 1.16]{bal:21}
\end{proof}

\begin{proposition}
\label{distribution}
The orbit distribution of the action map $\rho_C$ satisfies
\[\im \rho_C = L_C \cap \ker(p_{Z*}\circ\Phi_*).\]
\end{proposition}

\begin{proof}
{We begin by showing the forward inclusion. Let $(v, \alpha) \in B_{S,\gamma}$ and let $(u,v)= \rho_C(v, \alpha)$, so that $u$ is the unique element of $TM$ satisfying
\[\Phi_*u=v \quad \text{and} \quad (u, \Phi^*\alpha) \in L_M.\]
Then $(u,v) \in \ker (p_{Z*}\circ\Phi_*)$. Moreover,
\[p_M^*\Phi^*\alpha = \Phi^*p_X^*\alpha = \Phi^*\gamma^\flat(v),\]
so
\[((u,v), 0) = ((u,v), p_M^*\Phi^*\alpha - \Phi^*\gamma^\flat(v)) \in L_C,\]
and therefore $(u,v)\in\ker L_C$.}

{Conversely, let $(u,v) \in L_C \cap \ker(p_{Z*}\circ\Phi_*)$. Then there exists a 1-form $\beta\in T^*M$ such that
\[(u,\beta)\in L_M\quad\text{and}\quad p_M^*\beta=\Phi^*(\gamma^\flat(v)),\]
and the clean intersection condition guarantees that $\beta=\Phi^*\alpha$ for some covector $\alpha$ in $T^*X$. Since $\Phi$ is f-Dirac we have
\[(\Phi_*u,\alpha)=(v,\alpha)\in L_X,\]
and since $v$ is in the kernel of $p_{Z*}$ we get $(v,\alpha)\in B_{S,\gamma}$. Therefore 
\[(u,v)=\rho_C(v,\alpha).\qedhere\]}
\end{proof}

\begin{theorem}
\label{main}
{Suppose that the orbit space $Q\coloneqq C/B_{S,\gamma}$ is a manifold with quotient map $\pi:C\longrightarrow Q.$
\begin{enumerate}[label=\textup{(\arabic*)}]
\item\label{8b2exmfb}
The pushforward
\[L_Q\coloneqq \pi_*L_C\]
is an $\eta_Q$-twisted Dirac structure on $Q$, where $\eta_Q$ is uniquely characterized by the condition
\[\pi^*\eta_Q=\Phi^*(p_X^*\eta_X + d\gamma).\]
\item\label{bq5k3wpc}
The composition $p_Z \circ \Phi : C \too Z$ descends to a strong Dirac map 
\[\pbar:Q\longrightarrow Z.\]
\item\label{8b2exmfa}
The $f$-Dirac map $\pi : (C, L_C) \too (Q, L_Q)$ is also b-Dirac.
\item\label{4s6edzna}
The nondegenerate leaves of $L_Q$ are the images under $\pi$ of the nondegenerate leaves of $L_C$. In particular, if $M$ is nondegenerate and $L_M=L_\omega$, then $L_Q$ is also nondegenerate, and the presymplectic form $\overline{\omega}$ determining it is uniquely characterized by the equation
\[\pi^*\overline{\omega} = p_M^*\omega - \Phi^*\gamma.\]
\end{enumerate}}
\end{theorem}

\begin{proof}
(1) By Proposition \ref{bgfkcgj1}\ref{tpcp6ia5}, to check that $L_Q$ is an $\eta_Q$-twisted Dirac structure on $Q$ it suffices to check that the orbit distribution $\im \rho_C$ on $C$ is a regular Dirac distribution contained in the kernel of 
\[\Phi^*(p_X^*\eta_X + d\gamma) = \Phi^*p_Z^*\eta_Z.\]
By Remark \ref{mzq8z1k7}, this follows from Proposition \ref{distribution}.

(2) Note that the composition $p_Z\circ\Phi$ is f-Dirac by Proposition \ref{restriction}. Moreover, by Proposition \ref{distribution} we have 
\[\im \rho_C \s \ker (p_Z\circ\Phi)_*\quad\text{ and }\quad\ker (p_Z\circ\Phi)_* \cap L_C \s \im \rho_C,\]
so Proposition \ref{bgfkcgj1}\ref{15i0o7u8} implies that $p_Z$ descends to a strong Dirac map.

Parts (3) and (4) are a direct consequence of parts \ref{gpxy5jhi} and \ref{r0dvmlu4} of Proposition \ref{bgfkcgj1}, since Proposition \ref{distribution} implies that $\im \rho_C$ is contained in the kernel of $L_C$.
\end{proof}

{The $Z$-level $(S,\gamma)$ is an ordinary reduction level of $X$ if and only if $Z$ is a single point. In this case we obtain the following immediate corollary, which shows that ordinary reduction levels can be viewed as pre-Poisson submanifolds of twisted Dirac manifolds.}

\begin{corollary}
{Suppose that $(S,\gamma)$ is a reduction level of $(X,L_X)$ and that $\Phi:(M,L_M)\too (X,L_X)$ is an strong Dirac map} {which intersects $S$ cleanly. If the orbit space $Q\coloneqq \Phi^{-1}(S)/B_{S,\gamma}$ is a manifold sitting in the diagram}
\begin{equation*}
\begin{tikzcd}[row sep=large, column sep=large]
\Phi^{-1}(S)			\arrow[hook, r, "\jmath"] \arrow[d, swap, "\pi"]			&M\\
Q		&,
\end{tikzcd}
\end{equation*}
{then the distribution
\[L_Q\coloneqq \pi_*\tau_{-\Phi^*\gamma}\jmath^*L_M\]
is the graph of a Poisson structure on $Q$.}
\end{corollary}

\begin{remark}
\label{strata2}
We make several remarks about these results:
\begin{enumerate}[leftmargin=20pt] 
\item Suppose that the Dirac manifold $(X,L_X)$ is integrable with presymplectic groupoid $\G \tto X$. If the Lie algebroid $B_{S,\gamma}$ integrates to a source-connected groupoid
\[\G_{S,\gamma}\tto S\]
which acts on $C$, the orbits of $\G_{S,\gamma}$ coincide with those of $B_{S,\gamma}$ and Theorem \ref{main} gives a Dirac structure on the quotient $C/\G_{S,\gamma}$. In the nondegenerate case, an extension of this result to actions of certain non-source-connected groupoids is proved in Section \ref{fifth}.
\item We can apply Theorem \ref{main} to the case when $S=\sqcup S_i$ is a union of subpaces of $X\times Z$, as in Examples \ref{faces} and \ref{alcove}. That is, suppose that there exist $2$-forms $\gamma_i\in\Omega^2(S_i)$ such that each $(S_i,\gamma_i)$ is a $Z$-level of $X$. If each of these $Z$-levels satisfies the hypotheses of Theorem \ref{main}, we obtain a reduced space
\[\bigsqcup (M\times_XS_i)/B_{S_i,\gamma_i}\]
which is a topological quotient of $M\times_XS$ and which is equipped with a continuous map to $Z$ that restricts to a strong Dirac map along every stratum. {In certain cases, such as those of Examples \ref{faces} and \ref{alcove}, this quotient has the structure of a stratified space in the sense of \cite{mat:73}.}
\item Lastly, the assumptions on the $Z$-level $(S,\gamma)$ could be relaxed by replacing the pullback condition 
\[p_Z^*\eta_Z=p_X^*\eta_X + d\gamma\]
by the less restrictive condition
\[p_T(B_{S,\gamma})\subset\ker(p_X^*\eta_X + d\gamma).\]
The resulting reduction produces a f-Dirac map $\pbar:Q\too Z$ which satisfies the second condition of \eqref{strong}, but not the first.
\end{enumerate}
\end{remark}

Suppose that $(Y,L_Y)$ is an $\eta_Y$-twisted Dirac structure and consider the product Dirac manifold $(X\times Y, L_X\times L_Y)$. If $S$ is a $Z$-level of $X$, then $S\times Y$ is a $Z$-level of $X\times Y$. Theorem \ref{main} then has the following immediate corollary, which we will use repeatedly to study quasi-Poisson reduction in Section \ref{qpsec}.

\begin{corollary}
\label{cormain}
{Let 
\[(\mu,\nu):M\too X\times Y\]
be a strong Dirac map such that the $Z$-level $S$ of $X$ intersects $\mu\times p_X$ cleanly, and let $C= M\times_XS$. Suppose that the orbit space $Q\coloneqq C/B_{S,\gamma}$ is a manifold with quotient map $\pi:C\longrightarrow Q.$ Then the pushforward
\[L_Q\coloneqq \pi_*L_C\]
is an $\eta_Q$-twisted Dirac structure on $Q$, where $\eta_Q$ is uniquely characterized by the condition
\[\pi^*\eta_Q=\mu_C^*(p_X^*\eta_X + d\gamma)+p_M^*\nu^*\eta_Y,\]
and the map $p_Z$ descends to give a strong Dirac map 
\[(\pbar,\nubar):Q\longrightarrow Z\times Y.\]}
\end{corollary} 

We conclude this section by giving a simple freeness criterion that determines when the $Z$-level $S$ is transverse to a strong Dirac map. It is a generalization of the classical fact that, when $M$ is a Hamiltonian $G$-space, the point $0\in\g^*$ is a regular value of the moment map if and only if the action of $G$ on the $0$-fiber is locally free. To state it, recall that the condition $\rho_C(v, \alpha)=0$ implies that $v = 0$, so we may view $\ker \rho_C$ as a subset of $T^*X$.

\begin{proposition}
\label{freeness}
We have
\[p_{X*}(TS) + \Phi_*(p_T(L_M))  = (\ker \rho_C)^\circ.\]
In particular, if $B_{S,\gamma}$ acts freely on $C$, then $S$ intersects $\Phi\times p_X$ transversally, and the converse also holds if $M$ is a nondegenerate Dirac manifold.
\end{proposition}
\begin{proof}
We have
\begin{align*}
(\ker \rho_C)^\circ &= \{u \in TX \mid \alpha(u) = 0\text{ for all }\alpha \in T^*X\text{ such that }\imath^*\alpha = 0\text{ and }(0, \Phi^*\alpha) \in L_M \}\\
					&= (TS^\circ \cap (\Phi^*)^{-1}(L_M \cap T^*M))^\circ\\
					&= (TS + \Phi_*(p_T(L_M)))^\circ,
\end{align*}
where the second equality follows from the identity $p_T(L_M)^\circ = L_M \cap T^*M$.
\end{proof}

%
%
%
%
%
%
%
%\newpage
\subsection{Familiar examples of reduction}
\label{2.4}
The reduction procedure of Theorem \ref{main} recovers a number of familiar constructions in symplectic and Poisson geometry.

\begin{example}(Symplectic reduction along a submanifold.)
Let $X$ be a Poisson manifold. A strong Dirac map to $X$ is simply a Poisson map $\Phi:M\longrightarrow X$ from a Poisson manifold $M$. By Example \ref{ex:poisson2}, whenever $S\subset X$ is a pre-Poisson submanifold which intersects $\Phi$ cleanly and such that the quotient 
\[Q=\Phi^{-1}(S)/B_S\]
is a manifold, Theorem \ref{main} equips $Q$ with a reduced Poisson structure as in \cite[Theorem 8.83]{cra.fer.mar:21}. In particular, if $M$ is a symplectic manifold, the reduced structure on $Q$ is also symplectic and was studied in \cite{cro.may:21}.
\end{example}

\begin{example}
\label{pointred}
(Dirac reduction at a point.)
Let $X$ be a Dirac manifold and let $\Phi:M\longrightarrow X$ be a strong Dirac map. In view of Example \ref{point}, whenever $x\in X$ is a regular value of $\Phi$ and the quotient
\[Q_x\coloneqq \Phi^{-1}(x)/B_x\]
is a manifold, we recover the usual Poisson structure obtained by reducing a strong Dirac map at a point \cite[Theorem 4.11]{bur.cra:05}.
\end{example}

\begin{example}
\label{orbitred}
(Dirac reduction along a leaf.)
Let $X$ be a twisted Dirac manifold and let $\O$ be a nondegenerate leaf of $X$ with presymplectic form $\omega$. By Example \ref{orbit}, $(\O,\omega)$ is a reduction level of $X$ with stabilizer subalgebroid $B_{\O,\omega}=L_{X\vert\O}.$ Given a strong Dirac map $\Phi:M\too X$, the reduced space
\[Q_\O\coloneqq \Phi^{-1}(\O)/L_{X\vert\O}\]
is Poisson. 

Moreover, fixing a point $x\in\O$ gives a local diffeomorphism
\[\psi:Q_x\too Q_\O\]
between the Dirac reduction of $M$ at $x\in X$ and the space $Q_\O$, which fits into the diagram
\begin{equation*}
\begin{tikzcd}[row sep=large, column sep=large]
\Phi^{-1}(x)\arrow[r, hook, "\imath"]\arrow[d,"q"]	&\Phi^{-1}(\O)\arrow[r, hook, "\jmath"]\arrow[d, "q"]	&M\\
Q_x\arrow[r, "\psi"]	&\,Q_\O.	&
\end{tikzcd}
\end{equation*}
Chasing this diagram, we get 
\begin{align*}
\psi^*q_{*}\tau_{\Phi^*\omega}\jmath^*L_M&=q_{*}\imath^*\tau_{\Phi^*\omega}\jmath^*L_M\\
		&=q_{*}\tau_{\Phi^*\omega}\imath^*\jmath^*L_M\\
		&=q_{*}\imath^*\jmath^*L_M
\end{align*}
where the first equality follows from \cite[Lemma 1.16]{bal:21} and the third from the fact that the restriction of $\Phi^*\omega$ to $\Phi^{-1}(x)$ vanishes. In other words, $\varphi$ is a local Dirac diffeomorphism between $Q_\O$ and the reduction $Q_x$ of Example \ref{pointred} which intertwines the Dirac structures 
\[q_{*}\tau_{\Phi^*\omega}\jmath^*L_M\quad\text{and}\quad q_{*}\imath^*\jmath^*L_M.\]
\end{example}

In the special case of Hamiltonian Poisson and quasi-Poisson manifolds, Example \ref{orbitred} specializes to the usual Hamiltonian reduction along group orbits.

\begin{example}
(Marsden--Weinstein reduction along a coadjoint orbit.)
Let $\g$ be a Lie algebra and let $\mathcal{O}$ be a coadjoint orbit in $\g^*$ with symplectic form $\omega$. By Example \ref{orbit}, $(\O,\omega)$ is a reduction level of $\g^*$ with stabilizer subalgebroid
\[B_{\O,\omega}=\O\rtimes\g.\]
A strong Dirac map to $\g^*$ is given by a Hamiltonian Poisson $\g$-manifold $M$ with moment map $\Phi:M\too\g^*$. If the space
\[\Phi^{-1}(\mathcal{O})/B_{\O,\omega}=\Phi^{-1}(\mathcal{O})/\mathfrak{g}\]
is a manifold then its reduced Poisson structure coincides with the usual Marsden--Weinstein reduction \cite{mar.wei:74} along a coadjoint orbit.
\end{example}

\begin{example}
(Quasi-Poisson reduction along a conjugacy class.)
Let $G$ be a Lie group whose Lie algebra carries a nondegenerate inner product. Then $G$ has a Cartan--Dirac structure, and we let $\mathcal{O}\subset G$ be a conjugacy class with quasi-Hamiltonian $2$-form $\omega$. By Example \ref{orbit}, $(\mathcal{O},\omega)$ is a reduction level of $G$ with stabilizer subalgebroid
\[B_{\mathcal{O},\omega}=\mathcal{O}\rtimes\g.\]
A strong Dirac map to $G$ is a Hamiltonian quasi-Poisson $G$-manifold $M$ equipped with a moment map $\Phi:M\longrightarrow G$. If the reduced space
\[\Phi^{-1}(\mathcal{O})/B_{\mathcal{O},\omega}=\Phi^{-1}(\mathcal{O})/\g\]
is a manifold, Theorem \ref{main} equips it with a Poisson structure which coincides with the quasi-Hamiltonian reduction of \cite[Theorem 5.1]{ale.mal.mei:98} and \cite[Theorem 6.1]{ale.kos.mei:02}.
\end{example}

\begin{example}
\label{steinslice}
(Steinberg slices in quasi-Poisson manifolds.)
Let $G$ be a simply-connected, semisimple complex Lie group equipped with the Cartan--Dirac structure. The \emph{Steinberg cross-section} of $G$ is an affine space $\Sigma\subset G$ which consists entirely of regular elements and intersects every regular conjugacy class of $G$ exactly once and transversally \cite[Theorem 1.4]{ste:65}. This slice is a Poisson transversal in $G$, so by Lemma \ref{ptrans} it is a reduction level with trivial stabilizer subalgebroid. Moreover, if $\Phi:M\longrightarrow G$ is a quasi-Poisson moment map we have
\[T\Sigma+\im\Phi_*\supset T\Sigma+p_T(L_G)=TG\]
and therefore $\Sigma$ is transverse to $\Phi$. It then follows from Theorem \ref{main} that the pullback of the Dirac structure on $X$ to the preimage
\[\Phi^{-1}(\Sigma)\]
is a Poisson structure, recovering \cite[Theorem 2.3]{bal:21}.
\end{example}

\begin{example}
\label{unired}
(Parabolic reduction.)
Let $G$ be a reductive complex Lie group with Lie algebra $\g$, let $\p$ be a parabolic subalgebra of $\g$ with nilradical $\uu$, and write $P$ for the corresponding parabolic subgroup of $G$. Example \ref{unips} shows that $\uu$ is a reduction level of $\g$. If $\Phi:M\too\g$ is a Hamiltonian $G$-space, Theorem \ref{main} gives a Poisson structure on the quotient
\[\Phi^{-1}(\uu)/P.\]
This structure coincides with the following familiar Hamiltonian reduction. Since $M$ is also Hamiltonian for the action of $P$, there is a commutative diagram of moment maps
\begin{equation*}
\begin{tikzcd}[column sep=tiny]
M\arrow[rrd, "\Phi"]\arrow[dd,swap,"\Phi_P"]	&&\\
					&&\g\arrow[ld]\\
\,\,\quad\p^*\cong&\hspace{-.15in}\g/\uu.	&
\end{tikzcd}
\end{equation*}
Taking Hamiltonian reduction relative to $\Phi_P$, we obtain a Poisson structure on the quotient
\[\Phi_P^{-1}(0)/P=\Phi^{-1}(\uu)/P,\]
which agrees with the Poisson structure given by Theorem \ref{main}. In particular, when $M=T^*G$ is the cotangent bundle of $G$, we obtain the usual symplectic structure on the cotangent bundle
\[T^*(G/P)\cong G\times_P\uu\]
of the partial flag variety $G/P$.
\end{example}

\begin{example}
\label{parared}
(Unipotent reduction.)
Keeping the notation of Example \ref{unired}, let $\p=\l+\uu$ be a Levi decomposition with projection
\[\varphi:\p\too\l.\]
By Example \ref{parabs}, $\p$ is an $\l$-level of $\g$. If $\Phi:M\too\g$ is a Hamiltonian $G$-space, Theorem \ref{main} shows that the quotient of $\Phi^{-1}(\p)$ by the action of $U$ is a Hamiltonian Poisson $L$-space with moment map 
\[\phibar:\Phi^{-1}(\p)/U\too \l.\] As before, we may restrict to the action of $U$ to obtain the commutative diagram of moment maps
\begin{equation*}
\begin{tikzcd}[column sep=tiny]
M\arrow[rrd, "\Phi"]\arrow[dd,swap,"\Phi_U"]	&&\\
					&&\g\arrow[ld]\\
\,\,\quad\uu^*\cong&\hspace{-.15in}\g/\p	&
\end{tikzcd}
\end{equation*}
and reduce along $\Phi_U$. This gives a Poisson structure on the quotient
\[\Phi^{-1}_U(0)/U=\Phi^{-1}(\p)/U\]
which agrees with the structure obtained from Theorem \ref{main}. When $M=T^*G$, we obtain in this way the usual symplectic structure on the cotangent bundle
\[T^*(G/U)\cong G\times_U\p\]
of the base affine space $G/U$, for which the right action of $L$ is Hamiltonian.
\end{example}

\begin{example}
\label{sympimp}
(Symplectic implosion.)
Let $\g$ be the Lie algebra of a compact Lie group as in Example \ref{faces}, and fix a choice of closed dominant Weyl chamber
\[\t^*_+=\bigcup_{J\subset\Delta}\sigma_J\subset\t^*.\]
The inclusion map $\imath:\t^*_+\hooklongrightarrow\t^*$ makes $\t^*_+$ a stratified $\t$-level in the sense of Remark \ref{strata2}. If $\Phi:M\too \g^*$ is a symplectic Hamiltonian $G$-space, the reduced space 
\[\bigsqcup_J\Phi^{-1}(\sigma_J)/[G_J,G_J]\]
is a union of manifolds \cite[Theorem 5.10]{gui.jef.sja:02}. In view of Remark \ref{strata1}, this gives it the structure of a stratified Hamiltonian $T$-space with moment map $\overline{\imath}$, recovering the \emph{imploded cross-section} of $M$ developed in \cite[Theorem 2.10]{gui.jef.sja:02}.
\end{example}

\begin{example}
\label{qhamimp} 
(Quasi-Hamiltonian implosion.)
Let $G$ be a simply-connected compact Lie group equipped with a Cartan--Dirac quasi-Poisson structure, and fix a choice of closed dominant alcove $\overline{\mathcal{A}}$ as in Example \ref{alcove}. Its image under the exponential map then decomposes as a disjoint union
\[\exp(\bar{\mathcal{A}})=\bigsqcup_\sigma S_J,\]
so that inclusion map $\imath:\exp(\bar{\mathcal{A}})\hooklongrightarrow T$ gives $\exp(\bar{\mathcal{A}})$ the structure of a stratified $T$-level. If $\Phi:M\longrightarrow G$ is a quasi-Hamiltonian $G$-space, the reduced space 
\[\bigsqcup_J\Phi^{-1}(S_J)/[G_J,G_J]\]
is once again a union of smooth manifolds \cite[Corollary 3.12]{hur.jef.sja:06}, so it has the structure of a stratified quasi-Hamiltonian $T$-space with moment map $\overline{\imath}$. This recovers the quasi-Hamiltonian imploded cross-section of \cite[Theorem 3.17]{hur.jef.sja:06}.\\
\end{example}

%
%
%
%
%
%
%
%\newpage
\section{Fiber products and universality}
\label{third}
Fiber products over strong Dirac maps can be viewed as preimages of reduction levels, and this perspective has a number of immediate applications. After introducing it in Section \ref{3.1}, we apply it in Section \ref{3.2} to obtain an interpretation of the quasi-Poisson reduction of a fusion product in terms of Dirac geometry. In Section \ref{3.3}, we show that the Dirac reduction $\G_S$ of the presymplectic groupoid $\G\tto X$ at a given level is a ``universal'' reduced space, in the sense that any reduction along a strong Dirac map $\Phi: M\too X$ can be obtained from $\G_S$ by a quotient of the fiber product $M\times_X\G_S$.

\subsection{Fiber products of strong Dirac maps}
\label{3.1}
If $(X,L_X)$ is an $\eta_X$-twisted Dirac manifold, once again we write $\overline{X}$ for the manifold $X$ equipped with the opposite Dirac structure $-L_X$. Then the product $X\times\overline{X}$ carries an $(\eta_X,-\eta_X)$-twisted product Dirac structure, and we let $X_\Delta$ be the diagonal submanifold of this space.

\begin{lemma}
The subspace $X_\Delta$ is a reduction level of $X\times\overline{X}$.
\end{lemma}
\begin{proof}
The pullback of $(\eta_X,-\eta_X)$ to $X_\Delta$ vanishes since the tangent bundle $TX_\Delta\subset TX\times T\overline{X}$ consists of diagonal pairs of tangent vectors. Moreover 
\[L_{X\times\overline{X}}=\{((u,\alpha),(v,\beta))\mid (u,\alpha),(v,-\beta)\in L_X\},\]
so it follows that 
\[B_\Delta\coloneqq L_{X\times\overline{X}}\cap (TX_\Delta\oplus TX_\Delta^\circ)=\{((u,\alpha),(u,-\alpha))\in L_{X\times\overline{X}}\}\]
has constant rank.
\end{proof}

If $\Phi:M\longrightarrow X$ and $\Psi:N\longrightarrow \overline{X}$ are two strong Dirac maps with the property that the diagonal $X_\Delta$ has clean intersection with $(\Phi,\Psi)$, then the fiber product $M\times_XN$ is smooth. If the quotient $(M\times_XN)/B_\Delta$ is a manifold, Theorem \ref{main} equips it with a Poisson structure. In particular, by letting $\Psi$ be the identity map we recover the canonical induced Poisson structure on the quotient $M/L_X$ \cite[Proposition 3.14]{bur.igl.sev:09}. 

\begin{proposition}
\label{quotients}
Let $\Phi:M\too X$ be a strong Dirac map, and suppose that the quotient $M/L_X$ is a manifold with quotient map $\pi:M\too M/L_X.$ 
\begin{enumerate}[label=\textup{(\arabic*)}]
\item The pushforward {$\pi_*(L_M-\Phi^*L_X)=\pi_*L_M$} is a Poisson structure on $M/L_X$.
\item If $M$ is nondegenerate, the symplectic leaves of this structure are the reductions
\[\Phi^{-1}(\O)/L_X\]
of $M$ along the reductions $(\O,\omega)$ given by the nondegenerate leaves of $X$.
\end{enumerate}
\end{proposition}

\begin{proof}
(1) Consider the strong Dirac map $(\Phi,\Id):M\times\Xbar\too X\times\Xbar$. The reduction level $X_\Delta\subset X\times\Xbar$ is transverse to $(\Phi,\Id)$, and fits into the diagram
\begin{equation*}
\begin{tikzcd}[row sep=large, column sep=large]
M\arrow{r}{(\Id,\Phi)}[swap]{\sim}\arrow[d,swap,"\Phi"]	&M\times_X\Xbar\arrow[r, hook, "\jmath"]\arrow[d, "\Phi_\Delta"]	&M\times\Xbar\arrow[d, "\Phi\times\Id"]\\
X\arrow{r}[swap]{\sim}						&X_\Delta		\arrow[r, hook]							&X\times\Xbar.
\end{tikzcd}
\end{equation*}
Here the right-hand square is Cartesian, and the horizontal isomorphisms in the left-hand square intertwine the actions of $L_X$ and $B_\Delta$, so that we obtain a second commutative diagram
\begin{equation*}
\begin{tikzcd}[row sep=large, column sep=large]
M\arrow{r}{(\Id,\Phi)}[swap]{\sim}\arrow[d,swap,"\pi"]	&M\times_X\Xbar\arrow[d, "\pi_\Delta"]\\
M/L_X\arrow{r}{\varphi}[swap]{\sim}						&(M\times_X\Xbar)/B_\Delta.
\end{tikzcd}
\end{equation*}
By Theorem \ref{main}, there is a Poisson structure on $M/L_X$ given by
\[\varphi^*\pi_{\Delta*}\jmath^*(L_M\times L_{\Xbar})=\pi_*(\Id,\Phi)^*\jmath^*(L_M\times L_{\Xbar}) =\pi_*(L_M-\Phi^*L_X),\]
where the first equality follows from \cite[Lemma 1.16]{bal:21}.

{Moreover, note that this Poisson structure agrees with the pushforward $\pi_*L_M$: since both Dirac structures are Poisson \cite[Proposition 3.14]{bur.igl.sev:09}, each $\alpha\in T^*X$ corresponds to unique points
\[(\pi_*v,\alpha)\in \pi_*(L_M-\Phi^*L_X)\quad\text{and}\quad (\pi_*w,\alpha)\in\pi_*L_M.\]
It follows from the left-hand side that $\pi^*\alpha=\beta-\Phi^*\gamma$ where $(v,\beta)\in L_M$ and $(\Phi_*v,\gamma)\in L_X$, and from the right-hand side that $(w,\beta-\Phi^*\gamma)\in L_M$. We obtain
\[(v-w, \Phi^*\gamma)\in L_M,\]
and therefore $v-w$ is in the image of the action of $L_X$. In particular, $v-w\in\ker\pi_*$, so $\pi_*v=\pi_*w$ and $\pi_*(L_M-\Phi^*L_X)=\pi_*L_M$.}

(2) Suppose now that $M$ is nondegenerate with presymplectic form $\Omega$, and that
\[X=\sqcup(\O,\omega)\]
is the partition of $X$ into nondegenerate leaves. Then the presymplectic foliation of $M\times\Xbar$ is given by
\[M\times\Xbar=\sqcup(M\times\O,\Omega-\omega),\]
and by Theorem \ref{main}\ref{4s6edzna} the nondegenerate leaves of $M/L_X$ are all of the form
\[\varphi^{-1}(\pi_\Delta(M\times_\O\overline{\O}))=\Phi^{-1}(\O)/L_X.\]
Considering the diagram
\begin{equation*}
\begin{tikzcd}[row sep=large, column sep=large]
\Phi^{-1}(\O)\arrow[r, hook, "{\jmath_\O}"]\arrow[d, swap, "\pi"]	&M\arrow[d, "\pi"]\\
\Phi^{-1}(\O)/L_X\arrow[r, hook, "{\imath_\O}"]			&M/L_X,
\end{tikzcd}
\end{equation*}
the Dirac structure on each leaf is 
\[\imath_\O^*\pi_*(L_M-\Phi^*L_X)=\pi_*\jmath_\O^*(L_M-\Phi^*L_X)=\pi_*\tau_{-\Phi^*\omega}\jmath^*L_M=\pi_*L_{\Phi^{-1}(\O)},\]
where $L_{\Phi^{-1}(\O)}$ is the Dirac structure defined on $\Phi^{-1}(\O)$ by \eqref{twistpull} when it is viewed as the pullback of the reduction level $(\O,\omega)$ of $X$. Therefore the symplectic structure on $\Phi^{-1}(\O)/L_X$ is precisely the one obtained in Example \ref{orbitred} by reduction at $(\O,\omega)$.
\end{proof}

In particular, when $\g$ is a Lie algebra, $X=\g^*$ is equipped with the Kirillov--Kostant--Souriau Poisson structure, and 
\[\Phi:M\too\g^*\]
is the moment map of a Hamiltonian $G$-manifold, Proposition \ref{quotients} gives a Poisson structure on the quotient $M/G$ which coincides with the usual Poisson structure obtained by restricting the Poisson bracket of $M$ to the subring of invariant functions. Similarly, when $X=G$ is a Lie group equipped with the Cartan--Dirac structure and 
\[\Phi:M\too G\]
is the moment map of a Hamiltonian quasi-Poisson $G$-manifold, Proposition \ref{quotients} recovers the usual Poisson structure on the quotient $M/G$ obtained by pushing forward the quasi-Poisson bivector.

%
%
%
%
%
%
%
%\newpage
\subsection{Hamiltonian and quasi-Hamiltonian reduction}
\label{3.2}
Suppose that $M$ and $N$ are Hamiltonian Poisson $\g$-manifolds with moment maps $\mu:M\too\g^*$ and $\nu:N\too\g^*.$ Then the Poisson structure on the product space $M\times N$ is Hamiltonian for the diagonal action of $\g$, and the corresponding moment map is
\begin{align*}
\Phi:M\times N&\longrightarrow \quad\g^*\\
	(m,n)&\longmapsto\mu(m)+\nu(n).
	\end{align*}
In this case the Hamiltonian reduction $\Phi^{-1}(0)/\g$ of $M\times N$ is precisely the Dirac reduction of the Poisson map
\[(\mu,-\nu):M\times N\too\g^*\times\overline{\g^*}\]
along the reduction level $\g^*_\Delta$.

This observation has the following quasi-Poisson counterpart. Let $G$ be a Lie group whose Lie algebra $\g$ carries an invariant nondegenerate symmetric bilinear form, and suppose that $(M,\pi)$ is a Hamiltonian quasi-Poisson $G\times G$-manifold with moment map
\[(\Phi_1,\Phi_2):M\too G\times G.\]
Once again letting $\{e_i\}\subset\g$ be an orthonormal basis, we define the $2$-tensor
\[\psi\coloneqq \frac{1}{2}\sum(e_i,0)\wedge(0,e_i)\in\wedge^2(\g\oplus\g).\]
The bivector 
\begin{equation}
\label{fusion}
\pi_{\text{fus}}\coloneqq \pi+\psi_M.
\end{equation}
defines a Hamiltonian quasi-Poisson structure on $M$ relative to the diagonal action of $G$, and the corresponding moment map is the product $\Phi_1\Phi_2$ of the original moment map components \cite[Section 5]{ale.kos.mei:02}. The triple
\[(M,\pi_{\text{fus}},\Phi_1\Phi_2)\]
is called the \emph{internal fusion} of the quasi-Poisson $G\times G$-manifold $M$.

If $(M,\pi_M,\Phi_M)$ and $(N,\pi_N,\Phi_N)$ are Hamiltonian quasi-Poisson $G$-manifolds, the product $M\times N$ is naturally a quasi-Poisson $G\times G$-manifold. The \emph{fusion product} of $M$ and $N$, denoted $M\circledast N$, is the Hamiltonian quasi-Poisson $G$-manifold 
\[\left(M\times N, (\pi_M+\pi_N)_{\text{fus}}, \Phi_M\Phi_N\right).\]
If $1$ is a regular value of the moment map $\Phi_M\Phi_N$ and the quotient
\[(\Phi_M\Phi_N)^{-1}(1)/G\]
is a manifold, quasi-Poisson reduction \cite[Theorem 6.1]{ale.kos.mei:02} gives this quotient a Poisson structure.

On the other hand, let $\overline{G}$ be the group $G$ equipped with the opposite Cartan--Dirac structure $-L_G$. Then the composition
\[\iota\circ\Phi_N:N\too\overline{G}\]
of the moment map $\Phi_N$ with the group inversion is a strong Dirac map, and we may consider the fiber product
\[M\times_GN=\left\{(m,n)\in M\times N\mid \Phi_M(m)=\Phi_N(n)^{-1}\right\}=(\Phi_M\Phi_N)^{-1}(1).\]
Its quotient by the action of $G$, if it exists, has a Poisson structure obtained by applying Theorem \ref{main} to the diagonal. 

\begin{proposition}\label{ooysgjbr}
The Dirac structure on the reduction of the strong Dirac map $M\times N\too G\times\overline{G}$ along the diagonal reduction level $G_\Delta$ agrees with the quasi-Poisson reduction of $M\circledast N$ at the identity.
\end{proposition}
\begin{proof}
Consider the inclusion $\jmath:M\times_GN\hooklongrightarrow M\times N$ and let $q:M\times_GN\too (M\times_GN)/G$ be the quotient map. To show that
\[q_*\jmath^*L_{M\circledast N}=q_*\jmath^*(L_M\times L_N),\]
it is sufficient to prove that
\[\jmath^*L_{M\circledast N}=\jmath^*(L_M\times L_N)\]
by showing that the left-hand side is contained in the right-hand side.

Suppose that $(v,\jmath^*\gamma)\in \jmath^*L_{M\circledast N}.$ According to \eqref{qpdirac} and \eqref{fusion}, this implies that 
\[v=\pi_M^\#(\alpha)+\pi_N^\#(\beta)+\rho_M(\xi-\xi_\beta)+\rho_N(\xi+\xi_\alpha)\]
and
\[\gamma=C^*_{M\circledast N}(\alpha+\beta)+(\Phi_M\Phi_N)^*\sigma(\xi).\]
Here $\alpha\in T^*M$, $\beta\in T^*N$, $\xi\in \g$, and we write
\[\xi_\alpha\coloneqq \sum\alpha(e_{iM})e_i\quad\text{and}\quad\xi_\beta\coloneqq \sum\beta(e_{iN})e_{i}.\]
It is then enough to show that the pullback $\jmath^*\gamma$ agrees with
\[\jmath^*(C^*_M(\alpha)+C^*_N(\beta)+\Phi_M^*\sigma(\xi-\xi_\beta)+\Phi_N\sigma(\xi+\xi_\alpha)).\]

The equality
\begin{equation}
\label{xyz}
\jmath^*\circ(\Phi_M\Phi_N)^*=0
\end{equation}
implies that
\[\jmath^*\gamma=\jmath^*C^*_{M\circledast N}(\alpha+\beta)+\jmath^*(\Phi_M\Phi_N)^*\sigma(\xi)=\jmath^*\alpha+\jmath^*\beta.\]
Moreover, as in the proof of \cite[Lemma 2.11]{bal:21}, composition \eqref{xyz} implies that $\xi_\alpha=-\xi_\beta$ and that
\[\jmath^*\Phi_N^*(\theta_i^L)=-\jmath^*\Phi_M(\theta^R)\quad\text{and}\quad\jmath^*\Phi_N^*(\theta_i^R)=-\jmath^*\Phi_M(\theta^L).\]
We therefore obtain
\[\jmath^*(C^*_M(\alpha)+C^*_N(\beta))+\jmath^*(\Phi_M^*\sigma(\xi-\xi_\beta)+\Phi_N\sigma(\xi+\xi_\alpha))=\jmath^*\alpha+\jmath^*\beta,\]
proving the desired equality. It follows that $(v,\jmath^*\gamma)\in \jmath^*(L_M\times L_N)$, and therefore that the Dirac pullbacks $\jmath^*L_{M\circledast N}$ and $\jmath^*(L_M\times L_N)$ coincide.
\end{proof}

%
%
%
%
%
%
%
%\newpage
\subsection{Reduction of a presymplectic groupoid and universality}
\label{3.3}
Now suppose that $(X,L_X)$ is integrable, and let 
\[(\sss,\ttt):\G \tto X\]
be a source-connected presymplectic groupoid which integrates it. By \cite[Example 2.5]{bur.cra:09}, the map
\[(\sss,\ttt):\G\longrightarrow \overline{X}\times X\]
is a strong Dirac map. If $(S,\gamma)$ is a $Z$-level of $X$, $\HH$ is a souce-connected groupoid integrating $B_{S,\gamma}$, and the quotient $(\G\times_XS)/\HH$ is a manifold when the fiber product is taken along the target map, then from Corollary \ref{cormain} we obtain a Dirac structure on the reduced space 
\[\G_S\coloneqq (\G\times_XS)/\HH\]
such that the source map descends to a strong Dirac map
\[(\bar{\sss},\pbar):\G_S\longrightarrow \overline{X}\times Z.\]

Let $\Phi:M\longrightarrow X$ be a strong Dirac map, and consider the product Dirac structure
\[M\times\G_S\longrightarrow X\times \overline{X}\times Z.\]
Under the clean intersection hypothesis, the quotient $(M\times_X\G_S)/\G$ is a Dirac manifold with a strong Dirac map to $Z$. We will show that it is isomorphic to the Dirac reduction of $M$ at $(S,\gamma)$.

For every $x\in X$, let $e_x\in\G$ be the identity arrow at $x$. Then the natural b-Dirac map
\begin{align*}
M&\longrightarrow M\times_X\G \\	
m&\longmapsto (m, e_{\Phi(m)})
\end{align*}
restricts to a b-Dirac map
\begin{equation}
\label{restriction-2}
M\times_XS\longrightarrow M\times_X(\G\times_XS).
\end{equation}

\begin{proposition}
\label{univred}
The map \eqref{restriction-2} descends to a Dirac diffeomorphism
\[\overline{F}:(M\times_XS)/\HH\longrightarrow (M\times_X\G_S)/\G.\]
which commutes with the strong Dirac maps to $Z$. In particular, if $Z$ is a point then this map is an isomorphism of Poisson manifolds.
\end{proposition}
\begin{proof}
Suppose that $(m,e_{\Phi(m)})$ and $(m',e_{\Phi(m')})$ are conjugate under the action of $\G_\Delta\times \HH$. Then there exist $g\in \G$ and $h\in \HH$ such that 
\[\ttt(h)=\sss(g)=\Phi(m),\quad g\cdot m=m',\qquad\text{and}\qquad g\cdot e_{\Phi(m)}\cdot h= e_{\Phi(m')},\]
which implies that $g=h^{-1}$ and therefore that $g\in\HH$. It follows that $m$ and $m'$ are in the same $\HH$-orbit on $M\times_XS$, and therefore that $\overline{F}$ is injective.

Since $\G$ acts transitively on the fibers of the restriction
\[\sss:\G\times_XS\longrightarrow X,\]
every element of $M\times_X (\G\times_XS)$ is $\G_\Delta$-conjugate to an element of the form $(m,e_x)$. Moreover, if $(m,e_x)\in M\times_X (\G\times_XS)$, this implies that $\Phi(m)=x\in S$ and therefore $(m, e_x)$ is in the image of \eqref{restriction-2}. It follows that $\overline{F}$ is surjective. 

Since \eqref{restriction-2} is an injection on tangent spaces, it follows that $\overline{F}$ is bijective on tangent spaces, and therefore it is a diffeomorphism. It remains to show that it is a Dirac diffeomorphism. Consider the commutative diagram
\begin{equation*}
\begin{tikzcd}[column sep=large, row sep=large]
(M\times_XS)/\HH		\arrow[r, "F"]\arrow[rd, swap, "\overline{F}"]			&M\times_X\G_S \arrow[d]\\
						&\quad(M\times_X\G_S)/\G.
\end{tikzcd}
\end{equation*}
The vertical quotient map is b-Dirac by Theorem \ref{main}\ref{8b2exmfa}. The map $F$ descends from \eqref{restriction-2} via the diagram
\begin{equation*}
\begin{tikzcd}[column sep=large, row sep=large]
M\times_XS		\arrow[r]\arrow[d]		& M\times_X\G\times_XS\arrow[d]\\
(M\times_XS)/\HH			\arrow[r, "F"]				&M\times_X\G_S,
\end{tikzcd}
\end{equation*}
where the vertical arrows are quotients by the action of $\HH$. Since the vertical maps are f-Dirac and \eqref{restriction-2} is b-Dirac, it follows from \cite[Lemma 1.16]{bal:21} that $F$ is also b-Dirac. Then, since $\overline{F}$ is the composition of two b-Dirac maps, it is also b-Dirac and therefore a Dirac diffeomorphism.
\end{proof}

Because of Proposition \ref{univred}, we call $\G_S$ the \emph{universal reduced space} associated to the reduction level $(S,\gamma)$. It generalizes the symplectic universal reduced spaces of \cite[Section 4.3]{cro.may:21}, and specializes to a number of examples which already appear in the literature.

\begin{example}
\begin{enumerate}[leftmargin=20pt]
\item When $X=\g^*$ is the dual of a compact Lie algebra and $S=\t^*_+$ is the closure of the dominant Weyl chamber as in Example \ref{sympimp}, the universal reduced space is the universal imploded cross-section of \cite[Theorem 4.9]{gui.jef.sja:02}.

\item When $X=G$ is a compact Lie group and $S$ is the closed alcove $S=\overline{\mathcal{A}}$ of Example \ref{qhamimp}, Proposition \ref{ooysgjbr} implies that the universal reduced space is the quasi-Hamiltonian universal imploded cross-section of \cite[Theorem 4.6]{hur.jef.sja:06}.

\item When $X=G$ is a simply-connected semisimple complex Lie group and $S=\Sigma$ is the Steinberg cross-section from Example \ref{steinslice}, Proposition \ref{ooysgjbr} implies that the universal reduced space is the one-sided Steinberg slice of \cite[Proposition 2.17]{bal:21}.\\
\end{enumerate}
\end{example}

When $X=G$ is a Lie group equipped with the Cartan--Dirac structure, its presymplectic groupoid is the internal fusion double $D(G)$ of Example \ref{fusdouble}. If $(S,\gamma)$ is a $Z$-level of $G$, the corresponding universal reduced space is the $G$-space
\[D(G)_S=(G\times S)/\HH.\]
Since Proposition \ref{ooysgjbr} shows that Dirac reduction along the diagonal agrees with the usual quasi-Poisson reduction of a fusion product at the origin, we obtain the following quasi-Poisson version of Proposition \ref{univred}.

\begin{corollary}
Let $M$ be a Hamiltonian quasi-Poisson manifold with moment map $\Phi:M\longrightarrow G$. Then there is a Dirac diffeomorphism
\[(M\times_XS)/\HH\longrightarrow (M\circledast D(G)_S)\sslash_1G.\]
which commutes with the strong Dirac maps to $Z$. In particular, if $Z$ is a point then this map is an isomorphism of Poisson manifolds.\\
\end{corollary}

%
%
%
%
%
%
%
%\newpage
\section{Reductions of quasi-Poisson structures}
\label{qpsec}
We now show how Theorem \ref{main} applies in the setting of quasi-Poisson manifolds. In Section \ref{4.1} we outline a procedure for quasi-Poisson reduction relative to a certain class of subgroup, recovering one case of the quasi-Poisson reduction studied by Li-Bland and \v{S}evera \cite{lib.sev:15} and by \v{S}evera \cite{sev:15}. In Section \ref{4.2} we give a quasi-Poisson analogue to Whittaker reduction in which the levels are transversal slices that were introduced by Steinberg \cite{ste:65} and generalized by He and Lusztig \cite{lus.he:12} and by Sevostyanov \cite{sev:11}. In Section \ref{4.3} we use our reduction procedure to construct multiplicative versions of the open Moore--Tachikawa \cite{moo.tac:12} spaces introduced by Ginzburg and Kazhdan \cite{gin.kaz:21}, and we show that they satisfy the expected gluing laws.

\subsection{Quasi-Poisson reduction by a subgroup}
\label{4.1}
Let $G$ be a Lie group whose Lie algebra carries an invariant nondegenerate symmetric bilinear form, and equip it with the associated Cartan--Dirac structure.

\begin{definition}
A Lie subalgebra $\h$ of $\g$ is \emph{reducible} if its annihilator $\h^\perp$ is also a Lie subalgebra. We denote by $H$ and $H^\perp$ the connected subgroups of $G$ which integrate $\h$ and $\h^\perp$, and we call such subgroups \emph{reducible} as well.
\end{definition}

\begin{remark}
We will show that one can use Theorem \ref{main} to perform quasi-Poisson reduction relative to the action of a reducible subgroup of $G$. Applying this observation when $G$ is a complex reductive group produces Poisson and quasi-Poisson structures on a number of spaces that appear in geometric representation theory, including some that have been considered in work of Boalch \cite{boa:11}. All of the reductions we mention in this section can also be constructed, using the coisotropic subalgebra $(\h\times\h)+(\h^\perp)_\Delta$ of $\g\times\g$, as special cases of a more general reduction procedure for quasi-Poisson manifolds given by Li-Bland and \v{S}evera \cite[Theorem 3.1]{lib.sev:15}.
\end{remark}

The invariance of the bilinear form implies that $\h$ is reducible if and only if it satisfies the condition
\[[\h,\h^\perp]\s\h\cap\h^\perp.\]
In particular, any subgroup whose Lie algebra is Lagrangian is reducible---for example, the diagonal subgroup $G_\Delta\subset G\times \overline{G}$ is always a reducible subgroup. Moreover, $H$ is a reducible subgroup whenever $\h$ is coisotropic, as well as whenever $\h$ is isotropic and contains $[\h^\perp,\h^\perp]$.  When $G$ is complex and reductive, the following subgroups are examples of reducible subgroups:
\begin{itemize}[topsep=1.5pt]
\item parabolic subgroups
\item the unipotent radical of any parabolic subgroup
\item the derived subgroup of any parabolic subgroup
\end{itemize}

When $\h$ is reducible, the intersection $\h\cap\h^\perp$ is a normal Lie subalgebra of $\h$, and the pullback of the bilinear form of $\g$ descends to a nondegenerate form on the quotient
\[\bh\coloneqq \h/(\h\cap\h^\perp).\]
Therefore the Lie group
\[\bH\coloneqq H/(H\cap H^\perp)\]
can also be equipped with a Cartan--Dirac structure. We write $\varphi:H\too\bH$ for the quotient map.

\begin{lemma}
\label{plevels}
The subgroup $H$ is a $\bH$-level of $G$ whose stabilizer subalgebroid is the action Lie algebroid $H\rtimes\h^\perp$.
\end{lemma}
\begin{proof}
Consider the diagram
\begin{equation*}
\begin{tikzcd}[column sep=large, row sep=large]
H \arrow[hook]{r}{\imath} \arrow{d}{\varphi} & G \\
\overline{H}.
\end{tikzcd}
\end{equation*}
Since the Cartan $3$-form $\eta_G\in\Omega^3(G)$ vanishes along $\h\cap\h^\perp$, the pullback of $\eta_{\overline{H}}$ along $\varphi$ agrees with the pullback of $\eta_G$ along $\imath$.

We will now show that $\varphi_* \imath^*L_G = L_{\overline{H}}$. An element of $\varphi_*\imath^*L_G$ at a point $\varphi(h)$ is of the form $(\varphi_*v, \alpha)$, where $v = \xi^L-\xi^R$ and $\varphi^*\alpha=\sigma(\xi)$ for some $\xi \in \g$ satisfying
\[\xi-\Ad_{h}\xi \in \h \quad \text{ and } \quad \xi+\Ad_{h}\xi \in (\h \cap \h^\perp)^\perp = \h + \h^\perp.\]
In particular, this implies that $\xi \in \h + \h^\perp$. We write
\[\xi = \xi_{\h} + \xi_{\h^\perp}\]
with respect to this decomposition, and we let $\overline{\xi}_\h$ be the image of $\xi_\h$ under the quotient map to $\bh.$ By reducibility, $\xi_{\h^\perp}-\Ad_{h} \xi_{\h^\perp} \in \h \cap \h^\perp$ and so
\[\varphi_*(\xi^L-\xi^R)=\overline{\xi}_\h^L-\overline{\xi}_\h^R.\]
Similarly, $\xi_{\h^\perp}+\Ad_{h}\xi_{\h^\perp} \in \h^\perp$ and therefore
\[\alpha =\sigma(\overline{\xi}_\h).\]
It follows that $(\varphi_*v, \alpha) \in L_{\overline{H}}$, and we obtain the desired equality of Dirac structures. 

At any point $h\in H$,
\begin{align*}
B_{H}&=L_G\cap (T(H\cap H^\perp)\oplus TH^\circ)\\
		&=\left\{(\xi^L-\xi^R,\sigma(\xi))\mid \xi-\Ad_h\xi\in\h\cap\h^\perp, \xi+\Ad_h\xi\in\h^\perp\right\}\\
		&=\left\{(\xi^L-\xi^R,\sigma(\xi))\mid \xi\in\h^\perp\right\},
		\end{align*}
where the last equality follows from the reducibility condition. Therefore the rank of this distribution is constant, and under the identification of Example \ref{cd} it coincides with the action Lie algebroid
\[B_{H}\cong H\rtimes\h^\perp.\]
This implies that $H$ is a $\bH$-level of $G$.
\end{proof}

Applying Corollary \ref{cormain} in this setting, we obtain the following proposition. 

\begin{proposition}
\label{qpred}
Let $(\mu,\nu):M\too K\times G$ be a quasi-Poisson $K\times G$-manifold, and suppose that $\h$ is a reducible Lie subalgebra of $\g$ whose integrating subgroup $H\subset G$ intersects $\nu$ cleanly. If the quotient
\[Q=\nu^{-1}(H)/H^\perp\]
is a manifold, it has a natural quasi-Poisson $K\times\bH$-structure and the maps $\varphi$ and $\mu$ descend to a quasi-Poisson moment map
\[(\mubar,\overline{\varphi}):Q\too K\times\bH.\]
\end{proposition}

Suppose now that $G$ is a complex reductive group, let $P\subset G$ be a parabolic subgroup, and let $U\subset P$ be its unipotent radical. Then $P^\perp=U$, so $P$ and $U$ are both reducible and we may apply Proposition \ref{qpred}.

\begin{example}
The subspace $U$ of $G$ is a reduction level. If $(\mu,\nu):M\too K\times G$ is a quasi-Poisson $K\times G$-manifold, the reduction
\[\nu^{-1}(U)/P\too K\]
is a quasi-Poisson $K$-manifold. In particular, reducing the internal fusion double $D(G)$ of Example \ref{fusdouble} gives a quasi-Hamiltonian $G$-structure on $G\times_UP$. The corresponding moment map is
\begin{align*}
G\times_PU &\too G\\
[g:u]&\longmapsto gug^{-1},
\end{align*}
and this space is a multiplicative counterpart of the cotangent bundle of the partial flag variety $G/P$ described in Example \ref{unired}.
\end{example}

\begin{example}
\label{parabolicred}
Let $L$ be the Levi quotient of $P$ and let
\[\varphi:P\too L\]
be the quotient map. Then $P$ is an $L$-level of $G$ and, for any quasi-Poisson $K\times G$-manifold $(\mu,\nu):M\too K\times G$, the reduction
\[\nu^{-1}(P)/U\too K\times L\]
is a quasi-Poisson $K\times L$-manifold. In particular, applying this to the internal fusion double $D(G)$ we obtain the space $G\times_UP$, which has a quasi-Hamiltonian $G\times L$-structure with moment map
\begin{align*}
G\times_UP&\too G\times L\\
	[g:p]&\longmapsto (gpg^{-1}, pU).
	\end{align*}
This quasi-Hamiltonian manifold, which was introduced in \cite[Theorem 9]{boa:11}, is a multiplicative analogue of the cotangent bundle of the quasiaffine variety $G/U$ described in Example \ref{parared}.
\end{example}

Taking a quotient by the action of $\bH$ and applying Proposition \ref{quotients} gives a further corollary to Proposition \ref{qpred}.

\begin{corollary}
\label{qpred2}
Let $(\mu,\nu):M\too K\times G$ be a quasi-Poisson $K\times G$-manifold, and let $\h$ be a reducible Lie subalgebra of $\g$ whose integrating subgroup $H\subset G$ intersects $\nu$ cleanly. Suppose that the quotient 
\[Q=\nu^{-1}(H)/HH^\perp\]
is a manifold. 
\begin{enumerate}[label=\textup{(\arabic*)}]
\item The quotient $Q$ has a natural quasi-Poisson $K$-structure and the map $\mu$ descends to a quasi-Poisson moment map
\[\mubar:Q\too K.\]
\item For every conjugacy class $\O\subset\bH$, let $\O_{H\cap H^\perp}$ denote its preimage in $H$. Then 
\[\nu^{-1}(\O_{H\cap H^\perp})/HH^\perp\]
is a quasi-Poisson submanifold of $Q$.
\item If $M$ is quasi-Hamiltonian, then
\[Q=\bigsqcup\nu^{-1}(\O_{H\cap H^\perp})/HH^\perp\]
is the foliation of $Q$ by quasi-Hamiltonian leaves.
\end{enumerate}
\end{corollary}

\begin{example}
Let $G$ be a complex reductive group and $P$ be a parabolic subgroup with unipotent radical $U$ as above. Applying Corollary \ref{qpred2} to Example \ref{parabolicred} gives the bundle $G\times_PP$ the structure of a Hamiltonian quasi-Poisson $G$-space with moment map 
\begin{align*}
\Phi:G\times_PP&\longrightarrow G\\
	[g:p]&\longmapsto gpg^{-1}.
	\end{align*}
The partition of this space into quasi-Hamiltonian leaves is given by
\[G\times_PP=\bigsqcup G\times_P\O U,\]
where the right-hand side is a union over the conjugacy classes $\O$ of $L$. This recovers the quasi-Hamiltonian $G$-structures on $G\times_P\O U$ constructed in \cite[Section 4]{boa:11}.

In particular, when $P=B$ is a Borel subgroup, $U$ is the maximal unipotent subgroup of $G$ and $L=T$ is a torus. We obtain a quasi-Poisson structure on the multiplicative Grothendieck--Springer resolution
\[G\times_BB=\bigsqcup G\times_BtU\]
with quasi-Hamiltonian leaves given by twisted unipotent bundles indexed by elements $t\in T$.
\end{example}

If $\h$ is reducible, then the sum $\h+\h^\perp$ is also a reducible subalgebra of $\g$. If we make the additional assumption that $[\h^\perp,\h^\perp]\subset\h$, then the subgroup $HH^\perp$ together with the map
\[\psi:HH^\perp\too HH^\perp/H^\perp\cong \bH\]
gives another type of reduction level of $G$, which coincides with that of Lemma \ref{plevels} when $\h^\perp$ is contained in $\h$. Parabolic subalgebras of a complex reductive Lie algebra, their commutator subalgebras, and maximal nilpotent subalgebras are all examples of reducible subalgebras satisfying this additional condition.

\begin{lemma}
Suppose that $\h$ is reducible and that 
\begin{equation}
\label{perps}
[\h^\perp,\h^\perp]\subset\h.
\end{equation} 
Then the pair $HH^\perp$ is an $\bH$-level of $G$ whose stabilizer subalgebroid is the action Lie algebroid $HH^\perp\rtimes(\h\cap\h^\perp)$.
\end{lemma}
\begin{proof}
Factoring the map $\psi$ gives the diagram
\begin{equation*}
\begin{tikzcd}[]
&HH^\perp \arrow[ld,swap, "\varphi"]\arrow[hook]{rr}{\imath} \arrow{dd}{\psi} && G \\
HH^\perp/(H\cap H^\perp)\arrow[rd, swap, "q"]&&&\\
&\overline{H},&&
\end{tikzcd}
\end{equation*}
where $\varphi$ is the map obtained from viewing $HH^\perp$ as a reducible subgroup as in Lemma \ref{plevels}. Both quotients of $HH^\perp$ inherit invariant nondegenerate symmetric bilinear forms from $G$, so that we can equip them with Cartan--Dirac structures for which the map $q$ is f-Dirac. We therefore have
\[\psi_*\imath^* L_G=q_*\varphi_*\imath^* L_G=L_{\bH},\]
where the second equality follows from Lemma \ref{plevels}. Moreover, since $(\h^\perp,[\h^\perp,\h^\perp])=0$, the Cartan $3$-form vanishes along $H^\perp$ and therefore the pullback of $\eta_{\overline{H}}$ along $\varphi$ agrees with the pullback of $\eta_G$ along $\imath$.

At any point $g\in HH^\perp$,
\begin{align*}
B_{HH^\perp,\psi}&=L_G\cap (TH^\perp\oplus T(HH^\perp)^\circ)\\
		&=\left\{(\xi^L-\xi^R,\sigma(\xi))\mid \xi-\Ad_g\xi\in\h^\perp, \xi+\Ad_g\xi\in\h\cap\h^\perp\right\}\\
		&=\left\{(\xi^L-\xi^R,\sigma(\xi))\mid \xi\in\h\cap\h^\perp\right\},
		\end{align*}
where the last equation follows from assumption \eqref{perps}. Therefore 
\[B_{HH^\perp,\psi}\cong HH^\perp\rtimes (\h\cap\h^\perp)\]
has constant rank, and so we conclude that $HH^\perp$ is an $\bH$-level of $G$.
\end{proof}

Applying Corollary \ref{cormain}, we obtain the following reduction.

\begin{proposition}
\label{qpred3}
Let $(\mu,\nu):M\too K\times G$ be a quasi-Poisson $K\times G$-manifold, and suppose that $\h$ is a reducible Lie subalgebra of $\g$ which satisfies \eqref{perps} and such that $HH^\perp\subset G$ intersects $\nu$ cleanly. If the quotient
\[Q=\nu^{-1}(HH^\perp)/H\cap H^\perp\]
is a manifold, it has a natural quasi-Poisson $K\times\bH$-structure and the maps $\psi$ and $\mu$ descend to a quasi-Poisson moment map
\[(\mubar,\overline{\psi}):Q\too K\times\bH.\]
\end{proposition}

We may once again use Proposition \ref{quotients} to quotient further by the action of $\bH$.
 
\begin{proposition}
\label{qpred4}
Let $(\mu,\nu):M\too K\times G$ be a quasi-Poisson $K\times G$-manifold, and let $\h$ be a reducible Lie subalgebra of $\g$ which satisfies \eqref{perps} and such that $HH^\perp\subset G$ intersects $\nu$ cleanly. Suppose that the quotient 
\[Q=\mu^{-1}(HH^\perp)/H\]
is a manifold. 
\begin{enumerate}[label=\textup{(\arabic*)}]
\item The quotient $Q$ has a natural quasi-Poisson $K$-structure and the map $\mu$ descends to a quasi-Poisson moment map
\[\mubar:Q\too K.\]
\item For every conjugacy class $\O\subset\bH$, let $\O H^\perp$ denote its preimage in $HH^\perp$. Then 
\[\nu^{-1}(\O H^\perp)/H\]
is a quasi-Poisson submanifold of $\nu^{-1}(HH^\perp)/H$.
\item If $M$ is quasi-Hamiltonian, then
\[Q=\bigsqcup\nu^{-1}(\O H^\perp)/H\]
is the foliation of $Q$ by quasi-Hamiltonian leaves.
\end{enumerate}
\end{proposition}

\begin{example}
Once again let $G$ be a complex reductive group and let $P$ be a parabolic subgroup with unipotent radical $U$. Consider the reducible subgroup $H=[P,P]$. In this case we have
\[H^\perp=Z_LU, \quad HH^\perp=P, \quad H\cap H^\perp=U,\quad\text{and}\quad \bH=[L,L],\]
where $Z_L$ is the center of the Levi component of $P$. We may then reduce the fusion product $D(G)\too G\times G$ as in Proposition \ref{qpred4}. This gives a quasi-Poisson $G$-structure along the moment map
\begin{align*}
G\times_{[P,P]}P&\too G\\
		[g:p]&\longmapsto gpg^{-1}.
		\end{align*}
The quasi-Hamiltonian leaves of this manifold are indexed by the conjugacy classes of $[L,L]$ and the corresponding foliation is given by
\[G\times_{[P,P]}P=\bigsqcup G\times_{[P,P]}\O Z_LU.\]
\end{example}

%
%
%
%
%
%
%
%\newpage
\subsection{Quasi-Poisson Whittaker reduction}
\label{4.2}
Now let $G$ be a simply-connected, semisimple complex group equipped with the Cartan--Dirac structure and write $\g$ for its Lie algebra. Let $T$ be a fixed choice of maximal torus with Lie algebra $\t$, and let $W$ be the Weyl group of $G$ associated to $T$. 

Following Sevostyanov \cite{sev:11}, fix a Weyl group element $w\in W$ and let $\p$ be the associated parabolic subalgebra of $\g$. We write $\uu$ for its nilradical, $\l$ for its Levi component, and $P$, $U$, and $L$ for the corresponding subgroups of $G$. We let $U^-$ be the nilradical of the opposite parabolic subgroup $P^-$. 

Let $w\in N_G(T)$ be a normal representative of the Weyl group element $w$, which we denote by the same symbol, and let
\[Z=\{g\in L\mid wgw^{-1}=g\}\]
be the fixed-point subgroup of $w$ under the conjugation action. We write $\z$ for its Lie algebra. Then $[L,L]$ is contained in $Z$, and the intersection of $\z$ with $\t$ is precisely the fixed-point set $\t^w$. Moreover, let
\[U_w=\{u\in U\mid wuw^{-1}\in U^-\}=U\cap w^{-1}U^-w\]
be the subgroup of $U$ generated by the root spaces whose sign is flipped by $w$. This subgroup has dimension equal to the length of $w$, and it is normalized by $Z$. The \emph{Sevostyanov slice} through $w^{-1}$ is the product
\[\Sigma\coloneqq U_wZw^{-1}.\]
This slice is transverse to each conjugacy class that it meets \cite[Proposition 2.3]{sev:11}, and the conjugation action \cite[Proposition 2.1]{sev:11} gives an isomorphism
\begin{equation}
\label{target}
U\times\Sigma\xlongrightarrow{\sim} UZw^{-1}U\eqqcolon \Theta.
\end{equation}
In particular, when $w$ is a Coxeter element, we recover the Steinberg slice of Example \ref{steinslice}.

Let $\c$ be the annihilator of $\t^w$ in $\t$. Then the direct sum decomposition
\begin{equation}
\label{ortho}
\g=(\uu+\z+\uu^-)\oplus\c,
\end{equation}
is orthogonal with respect to the Killing form \cite[Equation 2.35]{sev:11}. This implies that the restriction of the Killing form to $\z$ is nondegenerate, and we can equip $Z$ with the associated Cartan--Dirac structure. We will now show that the projection map
\begin{align*}
\varphi:\qquad\Sigma\quad\,&\longrightarrow\,\, Z\\
			uzw^{-1}&\longmapsto\,\, z
			\end{align*}
makes the Sevostyanov slice $\Sigma$ a $Z$-level of $G$.

\begin{proposition}
\label{sigma}
The slice $\Sigma$ is a $Z$-level of $G$ with trivial stabilizer subalgebroid.
\end{proposition}
\begin{proof}
Consider the diagram
\begin{equation*}
\begin{tikzcd}[column sep=large, row sep=large]
\Sigma \arrow[hook]{r}{\imath} \arrow{d}{\varphi} & G \\
Z.
\end{tikzcd}
\end{equation*}
First we show that the pullback of the $3$-form $\eta_G$ through $\imath$ agrees with the pullback of $\eta_Z$ through $\varphi$. At any point $g\in\Sigma$, let $V_1,V_2,V_3\in T_g\Sigma$, and write 
\[V_i = n_i^R + z_i^R\quad\text{ for }n_i \in \uu_w\,\text{ and }\,z_i\in\z.\]
Then, since $\uu_w$ is isotropic and $\uu_w$ and $\z$ are orthogonal under the Killing form,
\[\imath^*\eta_G(V_1,V_2,V_3)=(n_1+z_1, [n_2+z_2,n_3+z_3]) 
					=(z_1, [z_2,z_3])
					=\varphi^*\eta_{Z}(V_1,V_2,V_3).\]

Next we show that $\varphi_*\imath^*L_G=L_Z$, by showing that the first bundle is contained in the second. An element of $\varphi_*\imath^*L_G$ at the point $\varphi(g)$ is of the form $(\varphi_*v,\alpha)$, where $v=\xi^L-\xi^R$ and $\varphi^*\alpha=\sigma(\xi)$ for an element $\xi\in\g$ that satisfies
\begin{equation}
\label{conditions}
\Ad_{g}\xi-\xi\in\uu_w+\z\quad\text{and}\quad\Ad_{g}\xi+\xi\in\uu_w^\perp.
\end{equation}
In view of \eqref{ortho}, it follows that 
\[\xi\in\uu_w^\perp=\uu+\z+\c+(\uu^-\cap\Ad_{w^{-1}}\uu^-),\]
and we write $\xi=u+z+c+\bar{u}$ with respect to this decomposition.

The first condition of \eqref{conditions} implies that $c=0$, while the second implies that
\[\Ad_g(u)\in\uu+\z+(\uu^-\cap\Ad_{w^{-1}}\uu^-).\]
Decomposing into the direct sum
\[\uu=(\uu\cap \Ad_{w}\uu)+(\uu\cap \Ad_{w}\uu^-),\]
we see that \eqref{conditions} is satisfied only if $u\in \uu\cap \Ad_{w}\uu$. Therefore
\[\Ad_g\xi-\xi=(\Ad_gu-u)+(\Ad_gz-z)+(\Ad_g\bar{u}-\bar{u}).\]
where the first term is contained in $\uu$ and the third is contained in $\uu+\uu^-$. This implies that
\[\varphi_*v=z^L-z^R.\]
Similarly, 
\[\Ad_g\xi+\xi=(\Ad_gu+u)+(\Ad_gz+z)+(\Ad_g\bar{u}+\bar{u})\]
and we obtain
\[\alpha=\sigma(z).\]
It follows that $(\varphi_*v,\alpha)\in\ L_Z$, and therefore the two Dirac structures coincide.

At any point $g\in\Sigma$ the distribution $B_{\Sigma,\varphi}$ consists of points of the form $(\xi^L-\xi^R,\sigma(\xi))$ with the property that
\begin{equation}
\label{conditions2}
\Ad_{g}\xi-\xi\in\uu_w\quad\text{and}\quad\Ad_{g}\xi+\xi\in(\uu_w+\z)^\perp.
\end{equation}
Once again by decomposition \eqref{ortho},
\[\xi\in(\uu_w+\z)^\perp=\uu+\c+(\uu^-\cap \Ad_{w^{-1}}\uu)\]
and we write $\xi=u+c+\bar{u}$. The first condition of \eqref{conditions2} implies that $c=0$, while the second again implies that $u\in\uu\cap\Ad_w\uu.$ Therefore 
\[\Ad_g\bar{u}\in\uu^-\cap \Ad_{w^{-1}}\uu,\]
and by the second condition of \eqref{conditions2} this means that $\Ad_g\bar{u}-\bar{u}=0$ and
\[\Ad_g\xi-\xi=\Ad_gu-u\in\uu_w.\]
The transversality given by isomorphism \eqref{target} then implies that $\Ad_g\xi-\xi=0$. Therefore 
\[B_{\Sigma,\varphi}=L_G\cap (T\Sigma\oplus T\Sigma^\circ)\cap(\ker\varphi_*\oplus T^*G)=0\]
is the zero vector bundle.
\end{proof}

The target $\Theta$ of isomorphism \eqref{target} can be written as a direct product
\[\Theta=U_wZw^{-1}U,\]
and therefore has a natural projection $\psi:\Theta\longrightarrow Z$ onto the reductive group $Z$.
Moreover, since at every point $g\in\Theta$ we can express the tangent space as
\[T_g\Theta = \{x^R + z^R+y^L \mid x, y \in \uu, z\in \z\},\]
we can define a $2$-form $\gamma\in\Omega^2(\Theta)$ by
\[\gamma_g(x_1^R + z_1^R+y_1^L, x_2^R + z_2^R+y_2^L) = \tfrac{1}{2}( x_2+ z_2, \Ad_gy_1)-\tfrac{1}{2}(x_1+ z_1, \Ad_gy_2).\]
Since the Killing form is $G$-invariant, $\uu$ is isotropic, and $\z$ annihilates $\uu$, it follows that $\gamma$ is well-defined.

\begin{proposition}
The pair $(\Theta,\gamma)$ is a $Z$-level of $G$ with stabilizer subalgebroid $\Theta\rtimes\uu$.
\end{proposition}
\begin{proof}
Let $p:\Theta\too\Sigma$ be the projection map induced by \eqref{target}. Then the inclusion of $\Sigma$ into $\Theta$ is a section of $p$ and, since $\psi$ is invariant along conjugacy classes of $U$, $\psi=\varphi\circ p$. We obtain a diagram
\begin{equation*}
\begin{tikzcd}[row sep=large, column sep=large]
\Sigma\arrow[r, hook,swap, "\imath"]\arrow[rd, swap, "\varphi"]	&\Theta\arrow[l, bend right=30, swap, "p"]\arrow[r, hook, "\jmath"]\arrow[d, "\psi"]	&G\\
&Z.	&
\end{tikzcd}
\end{equation*}
First we show that the differential of $\gamma$ satisfies the first condition of Definition \ref{genlevels}. Fix a point $g\in\Theta$, let $V_1,V_2,V_3\in T_g\Theta$, and write $V_i = x_i^R + z_i^R+y_i^L$ for $x_i,y_i \in \uu$,  and $z_i\in\z$. Then
\begin{align*}
2d\gamma(V_1, V_2, V_3) &= 2V_1 \gamma(V_2, V_3) + 2\gamma(V_1, [V_2, V_3]) + c.p. \\
					&= ([x_1, x_2], \Ad_gy_3) + (\Ad_g[y_1, y_2], x_3) + c.p.
					\end{align*}
and, since $\z$ annihilates $\uu$,
\begin{align*}
2\jmath^*\eta(V_1, V_2, V_3) &= (\Ad_{g^{-1}}x_1 + y_1, [\Ad_{g^{-1}}x_2 + y_2, \Ad_{g^{-1}}x_3 +  y_3])+2(z_1,[z_2,z_3])\\
					&= -2d\gamma(V_1, V_2, V_3)+2\psi^*\eta_Z(V_1,V_2,V_3).
					\end{align*}

Now we show that $p_*\tau_{-\gamma}\jmath^*L_G=\imath^*\jmath^*L_G.$ Suppose that $(v,\alpha)$ is an element of the right-hand side. Then there exists a Lie algebra element $\xi\in\g$ such that $\xi^L-\xi^R$ is tangent to $\Theta$ and
\[p_*(\xi^L-\xi^R)=v\quad\text{and}\quad p^*\alpha=\sigma(\xi)_{\vert\Theta}-\gamma^\flat(\xi^L-\xi^R).\]
By \eqref{target}, there exists an element $u\in\uu$ such that $(\xi-u)^L-(\xi-u)^R$ is tangent to $\Sigma$. Writing $\xi'=\xi-u$, this means that
\[\xi'^L-\xi'^R=p_*(\xi^L-\xi^R)=v\quad\text{and}\quad(\xi'^L-\xi'^R,\sigma(\xi')_{\vert\Theta})\in\jmath^*L_G.\]
Moreover, for any $x\in\uu_w$ and $z\in\z$,
\begin{align*}
\sigma(\xi')(x^R+z^R)&= \sigma(\xi)(x^R+z^R)-(\Ad_gu,x+z)\\
				&=\sigma(\xi)(x^R+z^R)-\gamma(u^L-u^R,x^R+z^R)\\
				&=\sigma(\xi)(x^R+z^R)-\gamma(\xi^L-\xi^R,x^R+z^R)\\
				&=p^*\alpha(x^R+z^R),
				\end{align*}
where the third equality follows since $\gamma$ vanishes along $\Sigma$. Then 
\[\imath^*\jmath^*\sigma(\xi')=\imath^*p^*\alpha=\alpha,\]
and this implies that $(v,\alpha)=(\xi'^L-\xi'^R,\imath^*\jmath^*\sigma(\xi'))\in\imath^*\jmath^*L_G.$ We therefore obtain, using Proposition \ref{sigma}, that
\[\psi_*\tau_{-\gamma}\jmath^*L_G=\varphi_*p_*\tau_{-\gamma}\jmath^*L_G=\varphi_*\imath^*\jmath^*L_G=L_Z.\]

Now we compute the distribution 
\[B_{\Theta,\gamma}=L_G\cap\imath_*L_{\gamma}\cap (\ker\psi_*\oplus T^*G).\]
In view of Example \ref{cd}, a pair $(g,\xi) \in \Theta\times \g$ lies in $B_{\Theta, \gamma}$ if and only if
\[ \Ad_{g}\xi-\xi \in \Ad_{g} \uu + \uu\]
and
\[\gamma_g(\xi^L- \xi^R, x^R+ z^R+y^L)= \frac{1}{2}(\Ad_{g} \xi+\xi, x + z+\Ad_{g}y)\quad\text{ for all $x,y\in\uu, z\in\z$}.\]
The second condition simplifies to
\[(\xi, x+y+z)=0\quad\text{ for all $x,y\in\uu, z\in\z$},\]
which occurs if and only if $\xi\in\p^\perp=\uu$. Therefore $B_{\Theta,\gamma}=\Theta\rtimes\uu.$
\end{proof}

Now let $K$ be a complex Lie group whose Lie algebra has an invariant inner product, and let $M$ be a quasi-Poisson $G\times K$-manifold with moment map
\[(\mu,\nu):M\longrightarrow G\times K.\]
Since the affine spaces $\Sigma$ and $\Theta$ are transverse to the conjugacy classes of $G$ \cite[Proposition 2.3]{sev:11}, their preimages are smooth submanifolds of $M$. Moreover, in view of \cite[Proposition 2.1]{sev:11}, there is an isomorphism 
\begin{equation}
\label{steiso}
\chi:\mu^{-1}(\Sigma)\xlongrightarrow{\sim}\mu^{-1}(\Theta)/U.
\end{equation}
In particular, the quotient $\mu^{-1}(\Theta)/U$ is alway a manifold, and the main Theorem \ref{main} equips both sides of this isomorphism with quasi-Poisson $Z\times K$-structures.

\begin{theorem}
\label{whittaker}
The map \eqref{steiso} is an isomorphism of quasi-Poisson $Z\times K$-manifolds.
\end{theorem}
\begin{proof}
Since it is clear that $\chi$ intertwines the quasi-Poisson moment maps of the two spaces, by \cite[Proposition 3.24]{bur.cra:05} it is sufficient to show that $\chi$ is f-Dirac. Consider the commutative diagram
\begin{equation*}
\begin{tikzcd}[row sep=large]
\mu^{-1}(\Sigma)\arrow[r, hook, "\imath"]\arrow[rd, swap, "\chi"]	&\mu^{-1}(\Theta)\arrow[r, hook, "\jmath"]\arrow[d, "\pi"]	&M\\
&\mu^{-1}(\Theta)/U.	&
\end{tikzcd}
\end{equation*}
The quasi-Poisson structure on $\mu^{-1}(\Sigma)$ is given by the Dirac pullback
\[L_\Sigma\coloneqq\tau_{-\mu^*\eta_G}\imath^*\jmath^*L_M,\]
and the quasi-Poisson structure on the quotient $\mu^{-1}(\Theta)/U$ is the Dirac bundle
\[L_\Theta\coloneqq \pi_*\tau_{-\mu^*\eta_G}\jmath^*L_M.\]

Chasing the diagram, we see that
\begin{align*}
\chi^*L_\Theta=\imath^*\pi^*\pi_*\tau_{-\mu^*\eta_G}\jmath^*L_M &=\imath^*\tau_{-\mu^*\eta_G}\jmath^*L_M\\
			&=\tau_{-\mu^*\eta_G}\imath^*\jmath^*L_M=L_\Sigma.
			\end{align*}
Here the second equality follows from the fact that the quotient map $\pi$ is both b-Dirac and f-Dirac by Theorem \ref{main}\ref{8b2exmfa}. This implies that the isomorphism $\chi$ is b-Dirac and therefore also f-Dirac.
\end{proof}

%
%
%
%
%
%
%
%\newpage
\subsection{The multiplicative Moore--Tachikawa TQFT}
\label{4.3}

For each complex reductive group $G$, Moore and Tachikawa \cite{moo.tac:12} conjectured the existence of a certain 2-dimensional topological quantum field theory (TQFT) valued in a category of holomorphic symplectic varieties. Such a TQFT is a symmetric monoidal functor defined on the category $\mathrm{Cob}_2$ of $2$-bordisms, where objects are 1-dimensional compact manifolds, morphisms are cobordisms between them, and composition is given by gluing. Its target is the category $\mathrm{HS}$ whose objects are complex reductive groups, a morphism from a group $G$ to a group $H$ is an affine symplectic variety with a Hamiltonian action of $G \times H$, and two morphisms $M:G \too H$ and $N : H \too K$ are composed by taking the symplectic reduction of $M \times N$ by the diagonal $H$-action. The Moore--Tachikawa conjecture states that, for each complex reductive group $G$, there is a TQFT 
\[\mathcal{F}_G:\mathrm{Cob}_2\too\mathrm{HS}\]
that sends the cap to the symplectic manifold $G \times \mathcal{S}$, where $\mathcal{S}$ is a Kostant slice. 
The non-triviality of this conjecture is exemplified by the case $G = \operatorname{SL}(3, \C)$, where $\mathcal{F}_G$ maps the pair of pants to the closure of the minimal nilpotent orbit of the exceptional Lie algebra $E_6$.

Using the gluing laws, the functor $\mathcal{F}_G$ is completely determined by the symplectic varieties to which it maps $n$-punctured spheres, which are known as the \emph{Moore--Tachikawa varieties}. Candidates for these varieties were constructed in \cite{gin.kaz:21} by defining their coordinate ring explicitly, reducing the conjecture to the statement that these rings are finitely generated. In type $A$, the conjecture was then settled by identifying the spaces constructed in \cite{gin.kaz:21} with Coulomb branches of star shaped quiver gauge theories \cite{bra.fin.nak:19}. In general Lie type, open dense submanifolds of the Moore--Tachikawa varieties can be realized as Hamiltonian reductions relative to the action of a group scheme known as the universal centralizer \cite{gin.kaz:21} (see also \cite{bie:21} for an equivalent construction).

In \cite{cro.may:21} it was shown that the construction of \cite{gin.kaz:21} is a special case of a more general procedure of symplectic reduction along a submanifold. In this section, we formulate a multiplicative version of the Moore--Tachikawa conjecture where the morphisms in the target category are quasi-Hamiltonian $G \times H$-varieties. The conjecture then requires a functor which sends the cap to the quasi-Hamiltonian $G$-space $G \times \Sigma$, where $\Sigma$ is the Steinberg cross-section discussed in Example \ref{steinslice}. The quasi-Hamiltonian varieties corresponding to the $n$-punctured spheres under this conjectural functor are multiplicative analogues of the Moore--Tachikawa varieties. We use our main Theorem \ref{main} to construct open dense submanifolds of these spaces and show that they satisfy the expected gluing laws.

Let $G$ be a simply-connected complex semisimple group and let $\Gad$ be its adjoint form. Consider the double
\[\mathbf{D} \coloneqq G\times\Gad ,\]
which is the quotient of the internal fusion double $D(G)$ of Example \ref{fusdouble} by the left action of the center of $G$. The quasi-Hamiltonian structure on $D(G)$ descends to a quasi-Hamiltonian structure on $\D$. Moreover, since the conjugation action of $G$ on itself descends to an action of $\Gad$, the double $\D$ is in fact a twisted presymplectic groupoid integrating the Cartan--Dirac structure on $G$. For any element $h \in G$, we write
\[Z_{\ad}(h) \coloneqq \{a \in \Gad \mid aha^{-1} = h\}\]
for its isotropy group with respect to this action, and $\g_h$ for the corresponding Lie subalgebra of $\g$.

Let $m$ and $n$ be non-negative integers, let $\Sigma \s G$ be the Steinberg cross-section, and identify it with the closed subvariety 
\[\Sigma_{m, n} \coloneqq \left\{(g,h) \in G^m \times G^n \mid g_1 = \cdots = g_m = h_1^{-1} = \cdots = h_n^{-1} \in \Sigma\right\}.\]
via the closed embedding
\begin{align*}
\Sigma&\hooklongrightarrow \qquad G^m\times G^n \\
	s&\longmapsto (s,\ldots,s,s^{-1},\ldots,s^{-1}).
	\end{align*}

\begin{proposition}
The submanifold $\Sigma_{m, n}$ is a reduction level of $G^{m+n}$ with stabilizer subalgebroid
\[B_{m,n}\coloneqq \left\{(s, x, y) \in \Sigma_{m,n}\times \g^m \times \g^n \mid x_1 + \cdots + x_m = y_1 + \cdots + y_n\text{ and }x_i, y_i\in\g_s\right\}.\]
\end{proposition}

\begin{proof}
Since the restriction of the Cartan $3$-form $\eta_G$ to the Steinberg cross-section vanishes, we must compute the distribution
\[B_{m,n}=L_{G^{m+n}}\cap (T\Sigma_{m,n}\oplus T\Sigma_{m,n}^\circ).\]
Let $\iota:G\too G$ be the inversion map. At every point $s\in\Sigma_{m,n}$, the fiber of $B_{m,n}$ consists of the points $(x,y)\in\g^m\times\g^n$ such that
\[(x_1^L - x_1^R)_s = \cdots = (x_m^L - x_m^R)_s = \iota_*(y_1^L - y_1^R)_{s^{-1}} = \cdots = \iota_*(y_n^L - y_n^R)_{s^{-1}} \in T\Sigma,\]
and
\[\left(x^{L\vee} + x^{R\vee}, y^{L\vee} + y^{R\vee}\right) \in T\Sigma_{m, n}^\circ.\]

Because $\Sigma$ intersects each conjugacy class transversally, all the terms in the first condition must vanish and so we obtain
\[\Ad_s (x_i) = x_i\quad \text{ and} \quad\Ad_s (y_i) = y_i \quad\text{ for all } i.\]
The second condition then implies that $(x^L,y^L) \in T\Sigma_{m, n}^\circ$, and this occurs if and only if
\[x_1+\cdots+x_m=y_1+\cdots+y_m.\]
It follows that $B_{m,n}$ is the desired subalgebroid and therefore $\Sigma_{m,n}$ is a reduction level.
\end{proof}

Consider the source-connected subgroupoid of $\D^{m+n}$ given by
\[H_{m, n} \coloneqq \{(s, a, b) \in \Sigma_{m,n}\times \Gad^m \times \Gad^n \mid a_1 \cdots a_m = b_1 \cdots b_n\text{ and }a_i, b_i \in Z_{\ad}(s)\}\]
which integrates $B_{m,n}$. We may view $\D^{m+n}$ as a quasi-Hamiltonian $G^{m+n} \times G^{m+n}$-space with moment map
\begin{align*}
\D^{m+n}&\too G^{m+n}\times G^{m+n}\\
		(g,a)&\mtoo (aga^{-1},g^{-1}).
		\end{align*}
Applying Corollary \ref{cormain} and reducing at the reduction level $\Sigma_{m,n}$ of the second factor of $G^{m+n}$, we obtain the quotient
\[\mathfrak{Z}_{m, n} \coloneqq (\Sigma\times\Gad^{m+n})/H_{m,n}.\]
This quotient is always a manifold because $H_{m,n}$ is a closed subgroupoid of $\Sigma\times \Gad^{m+n}$ and therefore acts properly. Reduction gives it the structure of a quasi-Hamiltonian $G^{m + n}$-space with moment map
\begin{align*}
\Z_{m,n}\,\,&\too \qquad G^{m+n} \\
(s,a,b)&\mtoo (a_is^{-1}a_i^{-1},b_isb_i^{-1}),
\end{align*}
and we refer to it as an \emph{open multiplicative Moore--Tachikawa variety}. 

\begin{example}
Taking $m=0$ and $n=1$ we obtain the space
\[\mathfrak{Z}_{0, 1} =\Sigma\times\Gad,\]
which is a quasi-Hamiltonian $G$-space studied in \cite[Section 2.2]{bal:21}. Taking $m=n=1$ gives
\[\mathfrak{Z}_{1, 1}\cong G^{\mathrm{r}}\times\Gad,\]
which is the restriction of the groupoid $\D$ to the regular locus of $G$.
\end{example}

Under the conjectural multiplicative Moore--Tachikawa correspondence, the space $\mathfrak{Z}_{m,n}$ is expected to be an open dense subset of the variety associated with the genus-0 cobordism from $m$ circles to $n$ circles. Accordingly, the composition of $\mathfrak{Z}_{m, n}$ and $\mathfrak{Z}_{\tm, \tn}$ corresponding to gluing one of the $n$ and $\tm$ circles, which is given by quasi-Hamiltonian reduction of the fusion product relative to the diagonal $G$-action, should be isomorphic to $\mathfrak{Z}_{m + \tm - 1, n + \tn - 1}$. The rest of this section is devoted to the proof of this fact.

Consider the quasi-Hamiltonian moment maps
\[(\mu,\mu_0):\mathfrak{Z}_{m,n}\too G^{m+n-1}\times G\quad\text{and}\quad(\tmu,\tmu_0):\mathfrak{Z}_{\tm,\tn}\too G^{\tm+\tn-1}\times G,\]
where $\mu_0$ and $\tmu_0$ are defined by
\[\mu_0(s,a,b)=b_1sb_1^{-1}\quad\text{and}\quad\tmu_0(\ts, \tilde{a}, \tilde{b})=\tilde{a}_1\ts^{-1} \tilde{a}_1^{-1}.\]
Using these two maps to reduce along the diagonal, we see that the open Moore--Tachikawa spaces satisfy the following recursion.

\begin{theorem}
\label{mooretac}
There is a canonical isomorphism of quasi-Hamiltonian $G^{m + \tm - 1} \times G^{n + \tn - 1}$-spaces
\[\Z_{m+\tm-1,n+\tn-1}\cong\left(\Z_{m,n}\circledast\Z_{\tm,\tn}\right)\sll{1}G.\]
\end{theorem}

\begin{proof}
We first show that $\Z_{m+\tm-1,n+\tn-1}$ is isomorphic to 
\[(\Z_{m,n}\times_G\Z_{\tm,\tn})/G,\]
where the fiber product is taken along the maps $\mu_0$ and $\imath\circ\tmu_0$. Since $\Sigma$ is a slice for the conjugation action of $G$ on $G^{\mathrm{r}}$, 
\[\Z_{m,n}\times_G\Z_{\tm,\tn}= \left\{(s,a,b,\ts, \tilde{a}, \tilde{b})\in \Z_{m,n}\times\Z_{\tm,\tn}\mid s=\ts\text{ and } \tilde{a}_1^{-1} b_1 \in Z_{\ad}(s)\right\}.\]
We can therefore define the $G^{m + \tm - 1} \times G^{n + \tn - 1}$-equivariant map
\begin{align*}
F: \Z_{m + \tm - 1, n + \tn - 1} &\too \Z_{m,n}\times_G\Z_{\tm,\tn}\\
		(s,a, a',b,b')&\longmapsto(s,a,\bar{b},s,\bar{a}',b').
		\end{align*}
Here $a\in G^m, a'\in G^{\tm-1}, b\in G^{n-1}, b'\in G^{\tn}$, and for any tuple $h=(h_1,\ldots,h_k)$ of elements of $G$ we write $\bar{h}\coloneqq (1,h_1,\ldots, h_k).$ The map $F$ descends to an isomorphism
\[\overline{F}:\Z_{m + \tm - 1, n + \tn - 1} \too (\Z_{m,n}\times_G\Z_{\tm,\tn})/G.\]

To show that $\overline{F}$ is a Dirac diffeomorphism, consider the commutative diagram
\begin{equation*}
\begin{tikzcd}[column sep=large, row sep=large]
\Z_{m + \tm - 1, n + \tn - 1}		\arrow[r, "F"]\arrow[rd, swap, "\overline{F}"]			&\Z_{m,n}\times_G\Z_{\tm,\tn} \arrow[d]\\
						&(\Z_{m,n}\times_G\Z_{\tm,\tn})/G.
\end{tikzcd}
\end{equation*}
The vertical map is b-Dirac by Theorem \ref{main}\ref{8b2exmfa}, and the horizontal map $F$ descends from a diagram
\begin{equation*}
\begin{tikzcd}[column sep=large, row sep=large]
\Sigma\times\Gad^{m+\tm-1}\times \Gad^{n+\tn-1}		\arrow[r]\arrow[d]		& \Sigma\times\Gad^{m+n}\times_G \Gad^{\tm+\tn}\arrow[d]\\
\Z_{m + \tm - 1, n + \tn - 1}			\arrow[r, "F"]				&\Z_{m,n}\times_G\Z_{\tm,\tn}.
\end{tikzcd}
\end{equation*}
Here the vertical arrows are quotients by the actions of $H_{m + \tm - 1, n + \tn - 1}$ and $H_{m,n}\times H_{\tm,\tn}$, and are therefore f-Dirac. The top horizontal map is the restriction of a map
\[\D^{k-2}\,\,\longrightarrow\,\, \D^{k}\,\cong\, \D^{k-2}\times \D^2\]
which has the form
\[\,\,\,\,\,(h,c)\,\,\longmapsto\,\, ((h,c),(h_1,h_2,1,1)).\]
---since the quasi-Hamiltonian $2$-form on $\D^k$ is a sum of the quasi-Hamiltonian $2$-forms on $\D^{k-2}$ and $\D^2$, and the second of these vanishes along the identity section of the presymplectic groupoid $\D^2$, this map is also b-Dirac. It then follows from \cite[Lemma 1.16]{bal:21} that $F$ is also b-Dirac.

Then, since $\overline{F}$ is the composition of two b-Dirac maps, it is also b-Dirac and therefore a Dirac diffeomorphism. By Proposition \ref{ooysgjbr}, the quasi-Poisson structure on its target agrees with the structure of the reduced fusion product $\left(\Z_{m,n}\circledast\Z_{\tm,\tn}\right)\sll{1}G$, completing the proof.\qedhere\\
\end{proof}

%
%
%
%
%
%
%
%\newpage
\section{Global version and interpretation in shifted symplectic geometry}
\label{fifth}

We now give a global version of the main Theorem \ref{main} in the case of a reduction level $(S, \gamma)$ of $X$ and a strong Dirac map $\Phi:M\too X$ from a nondegenerate Dirac manifold $M$. That is, we allow for quotients by Lie groupoids which integrate the stabilizer subalgebroid $B_{S, \gamma}$ and are not necessarily source-connected. We then give an interpretation of this reduction procedure in shifted symplectic geometry, as an intersection of two Lagrangians.

\subsection{Global version}
Let $(X,L_X)$ be an $\eta_X$-twisted Dirac manifold with presymplectic groupoid $\G\tto X$, and write $\Omega$ for the corresponding presymplectic form on $\G$. A \emph{Hamiltonian $\G$-space} \cite{xu:04} is a nondegenerate Dirac manifold $M$ with presymplectic form $\omega\in\Omega^2(M)$ and an action of $\G$ along a map $\Phi:M\too X$ such that
\begin{itemize}[topsep=1.5pt]
\item the graph of this action, which is given by
\[\Gamma = \{(g, x, g \cdot x) \in \G \times M \times M \mid {\sss}(g) = \mu(x)\},\]
is isotropic with respect to $(\Omega, \omega, -\omega)$,
\item $d\omega = -\Phi^*\eta$, and
\item $\ker \Phi \cap \ker \omega = 0$.
\end{itemize}
The first bullet point implies that $\Phi$ is a forward Dirac map, and the last two are equivalent to the condition that $\Phi$ is a strong Dirac map.

This notion generalizes both Hamiltonian and quasi-Hamiltonian manifolds for the action of a group $G$. When 
\[\G=T^*G\tto\g^*\]
is the symplectic groupoid of the Kirillov--Kostant--Souriau Poisson structure on $\g^*$, $M$ is a Hamiltonian $\G$-space if and only if it is Hamiltonian for the induced action of the group $G$. When 
\[\G=D(G) \tto G\]
is the quasi-Hamiltonian internal fusion double which integrates the Cartan--Dirac structure on $G$, a Hamiltonian $\G$-space is the same thing as a quasi-Hamiltonian $G$-manifold.

Symplectic and quasi-Hamiltonian reduction also generalize to the case of Hamiltonian $\G$-spaces---at any point $x$ of $X$ which is a regular value of $\mu$, if the quotient 
\[M \sll{x} \G \coloneqq \mu^{-1}(x) / \G_x\]
is a manifold then it has a natural reduced symplectic structure \cite[Theorem 3.18]{xu:04}. If the isotropy group $\G_x$ is connected, this reduction is a special case of Theorem \ref{main}. It is therefore natural to try to extend Theorem \ref{main} by taking quotients with respect to groupoids that integrate the stabilizer subalgebroid $B_{S, \gamma}$ but that are not source-connected. To do this, we will need the following definition.

\begin{definition}
\label{gpoids}
Let $(S,\gamma)$ be a reduction level in $X$. A \emph{stabilizer subgroupoid} of $S$ in $\G$ is a Lie groupoid $\HH \tto S$ together with a Lie groupoid morphism $f : \HH \too \G$ which
\begin{itemize}[topsep=1.5pt]
\item induces a Lie algebroid isomorphism from $\Lie(\HH)$ to $B_{S, \gamma}$, and such that 
\item $f^*\Omega ={\ttt^*\gamma - \sss^*\gamma}$.
\end{itemize}
\end{definition}

If $X$ is a Poisson manifold, {$\gamma = 0$}, and $\HH$ is source-connected, the second condition of Definition \ref{gpoids} is automatically satisfied \cite[Proposition 7.2]{cat.zam:09}. It is also easy to see that, in general, this condition holds along the identity section of any groupoid.

The usual reduction of Hamiltonian $\G$-spaces then has the following generalization.

\begin{theorem}\label{0rue6qjh}
Let $(S, \gamma)$ be a reduction level in $X$ and let $\HH \tto S$ be a stabilizer subgroupoid, and suppose that 
\[(\Phi,\Psi):M\longrightarrow X\times Y\]
is a Hamiltonian $\G_X\times\G_Y$-space such that $\HH$ acts freely and properly on $\Phi^{-1}(S)$. Then $S$ intersects $\Phi$ transversally and
\[M \sll{S} \G_X \coloneqq \Phi^{-1}(S) / \HH\]
is a Hamiltonian $\G_Y$-space. The map $\Psi$ descends to a moment map $\overline{\Psi}:M \sll{S} \G_X\too Y$, and the corresponding $2$-form $\overline{\omega}$ is uniquely characterized by the property that
\[\pi^*\overline{\omega} = \jmath^*\omega - \Phi^*\gamma,\]
where $\pi$ is the quotient map from $\Phi^{-1}(S)$ to $M \sll{S} \G$ and $\jmath$ is the inclusion of $\Phi^{-1}(S)$ into $M$. In particular, if $Y$ is a point then $M \sll{S} \G_X$ is a symplectic manifold.
\end{theorem}

\begin{proof}
By Proposition \ref{freeness}, $S$ intersects $\Phi$ transversally and therefore $\Phi^{-1}(S)/\HH$ is a manifold. Once again let $C \coloneqq \Phi^{-1}(S)$ and consider the action groupoids
\[\begin{tikzcd}
\HH \ltimes C  \arrow[shift right=1, swap]{d}{\sss_C} \arrow[shift left=1]{d}{\ttt_C} \arrow{r}{k} & (\G_X \times \G_Y) \ltimes M \arrow[shift right=1, swap]{d}{\sss_M} \arrow[shift left=1]{d}{\ttt_M} \\
C \arrow{r}{\jmath} & M.
\end{tikzcd}\]
The condition that the graph of the action of $\G_X \times \G_Y$ on $M$ is isotropic is equivalent to 
\[\ttt_M^*\omega - \sss_M^*\omega = \mathrm{pr}_{\G_X \times \G_Y}^*(\Omega_X, \Omega_Y),\]
where we write $\mathrm{pr}_{\G} : \G \ltimes M \too \G$ for the projection. Let $f : \HH \too \G_X$ be the groupoid morphism corresponding to the stabilizer subgroupoid $\HH$ and consider the commutative diagram
\[\begin{tikzcd}
\HH \ltimes C \arrow[shift right=1, swap]{d}{\sss_C} \arrow[shift left=1]{d}{\ttt_C} \arrow{r}{\mathrm{pr}_\HH} & \HH \arrow[shift right=1, swap]{d}{\sss} \arrow[shift left=1]{d}{\ttt} \\
C \arrow{r}{\Phi} & S
\end{tikzcd}\]
We have
\begin{align*}
\sss_C^*\jmath^*\omega - \ttt_C^*\jmath^*\omega&= k^*(\sss_M^*\omega - \ttt_M^*\omega) \\
									&= -k^*\mathrm{pr}_{\G_X \times \G_Y}^*(\Omega_X, \Omega_Y) \\
									&= - \mathrm{pr}_{\HH}^*(f^*\Omega_X) \\
									&= \mathrm{pr}_{\HH}^*(\sss^*\gamma - \ttt^*\gamma) \\
									&= \sss_C^*\Phi_{C}^*\gamma - \ttt_C^* \Phi_{C}^*\gamma,
\end{align*}
and therefore
\[\sss_C^*(\jmath^*\omega - \Phi_{C}^*\gamma) = \ttt_C^*(\jmath^*\omega - \Phi_{C}^*\gamma).\]
It follows that $\jmath^*\omega - \Phi^*\gamma$ descends to a 2-form $\overline{\omega}$ on the quotient $C / \HH$. A simple diagram chase shows that 
\[\ttt_Q^*\overline{\omega} - \sss_Q^*\overline{\omega} = \mathrm{pr}_{\G_Y}^*\Omega_Y,\]
where $\sss_Q$ and $\ttt_Q$ are the source and target maps of the action groupoid $\G_Y \ltimes Q \tto Q$. Moreover, we compute that
\[\pi^* d\overline{\omega} = \jmath^*d\omega - \Phi_{C}^*d\gamma = -\jmath^*(\Phi^*\eta_X + \Psi^*\eta_Y) + \Phi_{C}^*\imath^*\eta_X = - \jmath^*\Psi^*\eta_Y = - \pi^* \overline{\Psi}^*\eta_Y,\] 
and therefore $d\overline{\omega} = - \overline{\Psi}^*\eta_Y$. Finally, by Proposition \ref{distribution},
\[\ker\overline{\omega}\cap\ker\overline{\Psi}=\pi_*(\ker(\jmath^*\omega-\Phi^*\gamma)\cap\ker\Psi)=\pi_*(\im\rho_C)=0.\]
Therefore $\overline{\omega}$ gives $C / \HH$ the structure of a Hamiltonian $\G_Y$-space.
\end{proof}

\begin{example}\cite[Theorem 3.16]{xu:04}
Let $\G_X \tto X$, $\G_Y \tto Y$, and $\G_Z \tto Z$ be twisted presymplectic groupoids, let $M$ be a Hamiltonian $\G_X \times \overline{\G_Y}$-space, and $N$ a Hamiltonian $\G_Y \times \overline{\G_Z}$-space. Then 
\[M \times N \too X \times \bar{Y} \times Y \times \bar{Z}\]
is a Hamiltonian $\G_X \times \overline{\G_Y} \times \G_Y \times \overline{\G_Z}$-space. By Section \ref{third}, the diagonal $Y_\Delta \s \overline{Y} \times Y$ is a reduction level with stabilizer subgroupoid the diagonal of $\overline{\G_Y} \times \G_Y$. It follows from Theorem \ref{0rue6qjh} that if $\G_Y$ acts freely and properly on $M$ and $N$, then 
\[(M \times_Y N)/\G_Y\]
is a Hamiltonian $\G_X \times \overline{\G_Z}$-space.
\end{example}

\begin{example}
If $\G$ is a symplectic groupoid and $\gamma = 0$, then Theorem \ref{0rue6qjh} recovers the notion of symplectic reduction along a submanifold developed in \cite[Theorem A(ii)]{cro.may:21}.
\end{example}

%
%
%
%
%
%
%
%\newpage
\subsection{Interpretation in shifted symplectic geometry}
When $Y$ is a point, Theorem \ref{0rue6qjh} can also be seen as a consequence of interpreting our construction in the framework of shifted symplectic geometry \cite{pan.toe.vaq.vez:13}, as an intersection of two Lagrangians in a 1-shifted symplectic stack. Recall that there is an equivalence between 1-shifted symplectic 1-stacks and twisted presymplectic groupoids up to Morita equivalence \cite[Section 1.2.3]{cal:20}. Namely, given a Lie groupoid 
\[\G \tto X,\]
a 1-shifted symplectic structure on the quotient stack $[X/\G]$ consists precisely of a 2-form $\Omega$ on $\G$ and a 3-form $\eta$ on $X$ satisfying the definition of a presymplectic groupoid. Two twisted presymplectic groupoids correspond to the same 1-shifted symplectic stack if and only if they are Morita equivalent in the sense of \cite{xu:04}.

One of the main existence results for shifted symplectic structures is through the notion of Lagrangian morphisms. Given an $n$-shifted symplectic derived stack $Z$ and two Lagrangians $L_1 \too Z$ and $L_2 \too Z$, the fibre product $L_1 \times_Z L_2$ (in the sense of derived algebraic geometry) is $(n-1)$-shifted symplectic \cite[Theorem 2.9]{pan.toe.vaq.vez:13}. In the case where $Z = [X/\G]$ is a 1-shifted symplectic 1-stack, the notion of Lagrangians has the following simple description. Let $\HH \tto S$ be a Lie groupoid together with a Lie groupoid morphism $f : \HH \too \G$, which gives a morphism of quotient stacks
\[[f] : [S/\HH] \too [X/\G].\]
An \emph{isotropic structure} on $[f]$ is a $2$-form $\gamma\in\Omega^2(S)$ such that $f^*\Omega = \ttt^*\gamma - \sss^*\gamma\quad\text{and}\quad -d\gamma = f^*\eta.$ This structure is \emph{Lagrangian} if the induced morphism
\begin{equation}
\label{ldm50cv0}
\mathbb{T}_f  \too \mathbb{L}_{[S/\HH]}
\end{equation}
is a quasi-isomorphism, where 
\[\mathbb{T}_f \coloneqq f^*\mathbb{T}_{[X/\G]}[-1] \oplus \mathbb{T}_{[S/\HH]}\]
is the relative tangent complex endowed with the differential obtained from the map 
\[\mathbb{T}_{[S/\HH]} \too f^*\mathbb{T}_{[X/\G]},\]
and $\mathbb{L}_{[S/\HH]}$ is the cotangent complex.

The notion of Hamiltonian $\G$-space has a simple interpretation \cite[Example 1.31]{cal:20} in this setting: given a groupoid action of $\G \tto X$ along a smooth map $\Phi : M \too X,$ a 2-form $\omega$ on $M$ gives it the structure of Hamiltonian $\G$-space if and only if $\omega$ is a Lagrangian structure on the morphism of quotient stacks
\[[\Phi] : [M/\G] \too [X/\G].\]
To obtain 0-shifted symplectic structures---that is, ordinary symplectic structures---it then suffices to find another Lagrangian morphism $L \too [X/\G]$ and take the fibre product 
\[L \times_{[X/\G]} [M/\G].\]
The main result of this section is to show that Lagrangian \emph{substacks} of $[X/\G]$ are precisely stablizer subgroupoids of reduction levels.

\begin{theorem}\label{ex4wuat6}
Let $[X/\G]$ be a 1-shifted symplectic 1-stack and $\HH \tto S$ be a Lie subgroupoid of $\G \tto X$. A 2-form $\gamma\in\Omega^2(S)$ is a Lagrangian structure on 
\[[S/\HH] \too [X/\G]\]
if and only if $(S, \gamma)$ is a reduction level and $\HH \tto S$ is a stabilizer subgroupoid of $(S, \gamma)$.
\end{theorem}

\begin{remark}
\label{abzwgl6w}
If $(S, \gamma)$ is a reduction level of $X$ and $\gamma$ extends to a $2$-form $\tilde{\gamma}\in\Omega^2(X)$, the stabilizer subalgebroid $B_{S,\gamma}$ is equal to the stabilizer subalgebroid of the reduction level $(S, 0)$ relative to the gauge-transformed Dirac structure $\tau_{\tilde{\gamma}}L_X.$ In the holomorphic or algebraic setting such a global extension may not exist, but this still allows us to reduce local results to the case where $\gamma = 0$, as we will do here.
\end{remark}

\begin{lemma}\label{19xp3wff}
Let $(S, \gamma)$ be a reduction level in $X$ and let $p_{T,B} : B_{S, \gamma} \too TS$ be the anchor map. Then the sequences
\[\begin{tikzcd}[row sep = -1mm]
0 \arrow{r} & B_{S, \gamma} \arrow{r} & L_X \cap (TS \oplus T^*X) \arrow{r} & \ker(p_{T,B}^*) \arrow{r} & 0,\\
& & \hspace{2cm}(v, \alpha) \arrow[mapsto]{r} & \imath^*\alpha - \gamma^\flat(v)
\end{tikzcd}\]
and
\[\begin{tikzcd}
0 \arrow{r} & B_{S,\gamma} \cap T^*X \arrow{r} & T^*X \arrow{r} & (\im p_T + TS)^* \arrow{r} & 0
\end{tikzcd}\]
are exact.
\end{lemma}

\begin{proof}
Since this is a pointwise result, by Remark \ref{abzwgl6w} we may assume that $\gamma = 0$. For the first sequence, we need to show that the map $\varphi : (v, \alpha) \longmapsto \imath^*\alpha$ is surjective. By identifying 
\[\ker(p_{T,B}) = p_T(B_S)^\circ / TS^\circ\]
and writing $p_{T^*}:L_X\too T^*X$ for the second projection, the image of $\varphi$ is 
\[(p_{T^*}(L_X \cap (TS \oplus T^*X)) + TS^\circ) / TS^\circ.\]
Since 
\[p_T(B_S)^\circ = TS^\circ + p_T(L_X \cap (TX \oplus TS^\circ))^\circ,\]
the statement now follows from the fact that
\[p_{T^*}(L_X \cap (TS \oplus T^*X)) = p_T(L_X \cap (TX \oplus TS^\circ))^\circ.\]

For the second sequence, we need to show that 
\[B_{S, \gamma} \cap T^*X = (\im p_T + TS)^\circ.\]
Note that $\alpha \in B_{S, \gamma} \cap T^*X$ if and only if $(0, \alpha) \in L_X$ and $\alpha \in TS^\circ$. By the isotropy of $L_X$, this is equivalent to $\alpha \in (\im p_T)^\circ$. 
\end{proof}

\begin{proof}[Proof of Theorem \ref{ex4wuat6}]
Suppose that $\gamma$ is a Lagrangian structure and let $L \coloneqq \Lie(\HH)$. Since $\imath^*\eta = -d\gamma$, it suffices to show that $L = B_{S, \gamma}$. The map of complexes \eqref{ldm50cv0} is then given by
\begin{equation}\label{64x0xucj}
\begin{tikzcd}[row sep=large]
L \arrow{r}{\alpha} \arrow{d} & L_{X\vert S} \oplus TS \arrow{r}{\beta} \arrow{d}{\gamma} & TX_{\vert S} \arrow{d}{\delta} \\
0 \arrow{r} & T^*S \arrow{r}{\varepsilon} & L^*,
\end{tikzcd}
\end{equation}
where the maps are defined by
\[\alpha(l) = (l, \rho_L(l)),\quad \beta(a, v) = v - \rho(a),\quad\text{and}\quad \gamma(a, v) = \imath^*\sigma(a) - \gamma^\flat(v),\]
$\delta$ is the dual of 
\[L \too L_{X\vert S} \xlongrightarrow{p_{T^*}} T^*X_{\vert S},\]
and $\varepsilon$ is the dual of the anchor map $p_{T,L}:L\too TS$.

The middle map on cohomologies $\ker \beta / \im \alpha \too \ker \varepsilon$ then reduces to
\begin{align}\label{ekskha1y}
\varphi : p_T^{-1}(TS) / L &\too \ker(p_{T,L}^*)\\
a&\mtoo \imath^*p_{T^*}(a) - \gamma^\flat(p_T(a)).\nonumber
\end{align}
The assumption $\ker \varphi = 0$ implies that $L = B_{S, \gamma}$. Therefore $S$ is a reduction level and $\HH$ is a stabilizer subgroupoid of $S$.

Conversely, suppose that $(S, \gamma)$ is a reduction level and $\HH \tto S$ is a stabilizer subgroupoid of $S$. We need to show that \eqref{64x0xucj} is a quasi-isomorphism. Since $L = B_{S, \gamma}$, Lemma \ref{19xp3wff} shows that \eqref{ekskha1y} is an isomorphism. Showing that the third map on cohomologies is an isomorphism amounts to the exactness of the sequence
\[\begin{tikzcd}
0 \arrow{r} & \im \rho + TS \arrow{r} & TX|_S \arrow{r} \arrow{r} & (\ker \rho_L)^* \arrow{r} & 0,
\end{tikzcd}\]
which is the dual of the second sequence in Lemma \ref{19xp3wff}.
\end{proof}

In the case that $Y$ is a point, Theorem \ref{0rue6qjh} then also follows from the observation that
\[M \sll{S} \G =  [S/\HH] \times_{[X / \G]} [M/\G].\]
Moreover, note that even in the absence of smoothness conditions we can still interpret $M \sll{S} \G$ as a 0-shifted symplectic derived stack.

\begin{remark}\label{8fwv7gf8}
The results of this section can be generalized to degenerate Dirac manifolds via coisotropic intersections in shifted Poisson geometry \cite{cptvv}.
Namely, a Hamiltonian action of a quasi-symplectic groupoid $\G \tto X$ on a Dirac manifold $(M, L_M)$ (in the sense of \cite{bur.cra:05}) gives a 1-shifted coisotropic structure on $[M/\G] \to [X/\G]$.
The intersection $M \sll{S} \G = [S/\HH] \times_{[X/\G]} [M/\G]$ is then 0-shifted Poisson \cite{melani-safranov-1,melani-safranov-2}.
More generally, if $(M, L_M)$ is acted on by $\G_X \times \G_Y$, then $M \sll{S} \G_X$ is a Hamiltonian $\G_Y$-space.
See \cite{mayrand:2025} for details.
\end{remark}

%
%
%
%
%
%
%
%\newpage
\bibliographystyle{plain}
\bibliography{biblio}

\begin{thebibliography}{CPTVV}

\bibitem[ABM]{ale.bur.mei:09}
A.~Alekseev, H.~Bursztyn, and E.~Meinrenken.
\newblock Pure spinors on {L}ie groups.
\newblock {\em Ast\'erisque}, 327:131--199, 2009.

\bibitem[AKM]{ale.kos.mei:02}
A.~Alekseev, Y.~Kosmann-Schwarzbach, and E.~Meinrenken.
\newblock Quasi-{P}oisson manifolds.
\newblock {\em Can. J. Math}, 54(1):3--29, 2002.

\bibitem[AMM]{ale.mal.mei:98}
A.~Alekseev, M.~Malkin, and E.~Meinrenken.
\newblock Lie group valued moment maps.
\newblock {\em J. Differential Geom.}, 48(3):445--495, 1998.

\bibitem[Bal]{bal:21}
A.~Balibanu.
\newblock Steinberg slices and group-valued moment maps.
\newblock {\em Adv. Math.}, 402, 2022.

\bibitem[Bie]{bie:21}
R.~Bielawski.
\newblock {On the Moore-Tachikawa varieties}.
\newblock \href{https://arxiv.org/abs/2104.05555}{arXiv:2104.05555}.

\bibitem[Boa]{boa:11}
P.~Boalch.
\newblock Riemann--{H}ilbert for tame complex parahoric connections.
\newblock {\em Transform. Groups}, 16(1):27--50, 2011.

\bibitem[BFM]{bra.fre.mar:20}
L.~Brambila, P.~Frejlich, and D.~Martinez Torres.
\newblock Coregular submanifolds and {P}oisson submersions.
\newblock \href{https://arxiv.org/abs/2010.09058}{arXiv:2010.09058}.

\bibitem[BFN]{bra.fin.nak:19}
A.~Braverman, M.~Finkelberg, and H.~Nakajima.
\newblock Ring objects in the equivariant derived {S}atake category arising
  from {C}oulomb branches.
\newblock {\em Adv. Theor. Math. Phys.}, 23(2):253--344, 2019.
\newblock Appendix by Gus Lonergan.

\bibitem[Bur]{bur:13}
H.~Bursztyn.
\newblock A brief introduction to {D}irac manifolds.
\newblock In {\em Geometric and topological methods for quantum field theory}.
  Cambridge Univ. Press, 2013.

\bibitem[BuCr1]{bur.cra:05}
H.~Bursztyn and M.~Crainic.
\newblock Dirac structures, momentum maps and quasi-{P}oisson manifolds.
\newblock In {\em The breadth of symplectic and Poisson geometry}. Progress in
  Mathematics, 2005.

\bibitem[BuCr2]{bur.cra:09}
H.~Bursztyn and M.~Crainic.
\newblock Dirac geometry, quasi-{P}oisson actions and ${D}/{G}$-valued moment
  maps.
\newblock {\em J. Differential Geom.}, 82:501--566, 2009.

\bibitem[BI\v{S}]{bur.igl.sev:09}
H.~Bursztyn, D.~Iglesias Ponte, and P.~\v{S}evera.
\newblock Courant morphisms and moment maps.
\newblock {\em Math. Res. Lett.}, 16(2):215--232, 2009.

\bibitem[Cal]{cal:20}
D.~Calaque.
\newblock Derived stacks in symplectic geometry.
\newblock In {\em New Spaces in Physics: Formal and Conceptual Reflections}.
  Cambridge University Press, 2020.

\bibitem[CPTVV]{cptvv}
D.~Calaque, T.~Pantev, B.~To\"{e}n, M.~Vaqui\'{e}, and G.~Vezzosi.
\newblock Shifted {P}oisson structures and deformation quantization.
\newblock {\em J. Topol.}, 10(2):483--584, 2017.

\bibitem[CaZa]{cat.zam:09}
A.~Cattaneo and M.~Zambon.
\newblock Coisotropic embeddings in {P}oisson manifolds.
\newblock {\em Trans. Amer. Math. Soc.}, 361(7):3721--3746, 2009.

\bibitem[CFM]{cra.fer.mar:21}
M.~Crainic, R.~Loja Fernandes, and I.~Marcut.
\newblock {\em Lectures on Poisson Geometry}.
\newblock American Mathematical Society, 2021.

\bibitem[CrMa]{cro.may:21}
P.~Crooks and M.~Mayrand.
\newblock Symplectic reduction along a submanifold.
\newblock {\em Compositio Math.}, 158(9):1878--1934, 2022.

\bibitem[GiKa]{gin.kaz:21}
V.~Ginzburg and D.~Kazhdan.
\newblock Algebraic symplectic manifolds arising in {S}icilian theories.
\newblock {\em forthcoming}.

\bibitem[Gua]{gua:11}
M.~Gualtieri.
\newblock Generalized complex geometry.
\newblock {\em Ann. Math.}, 174(1):75--123, 2011.

\bibitem[GJS]{gui.jef.sja:02}
V.~Guillemin, L.~Jeffrey, and R.~Sjamaar.
\newblock Symplectic implosion.
\newblock {\em Transform. Groups}, 7(2):155--184, 2002.

\bibitem[HeLu]{lus.he:12}
X.~He and G.~Lusztig.
\newblock A generalization of {S}teinberg's cross section.
\newblock {\em J. Amer. Math. Soc.}, 25(3):739--757, 2012.

\bibitem[HJS]{hur.jef.sja:06}
J.~Hurtubise, L.~Jeffrey, and R.~Sjamaar.
\newblock Group-valued implosion and parabolic structures.
\newblock {\em Amer. J. Math.}, 128(1):167--214, 2006.

\bibitem[Kos]{kos:79}
B.~Kostant.
\newblock The solution to a generalized {T}oda lattice and representation
  theory.
\newblock {\em Adv. in Math.}, 34(3):195--338, 1979.

\bibitem[Li\v{S}e]{lib.sev:15}
D.~Li-Bland and P.~\v{S}evera.
\newblock Symplectic and {P}oisson geometry of the moduli spaces of flat
  connections over quilted surfaces.
\newblock In {\em Mathematical Aspects of Quantum Field Theories}. Springer,
  2015.

\bibitem[MaWe]{mar.wei:74}
J.~Marsden and A.~Weinstein.
\newblock Reduction of symplectic manifolds with symmetry.
\newblock {\em Rep. Math. Phys.}, 5:121--130, 1974.

\bibitem[Mat]{mat:73}
J.~Mather.
\newblock Stratifications and mappings.
\newblock In {\em Dynamical Systems: Proceedings of a Symposium Held at the
  University of Bahia, Salvador, Brasil}. 1973.

\bibitem[May]{mayrand:2025}
M.~Mayrand.
\newblock Shifted coisotropic structures for differentiable stacks.
\newblock {\em Advances in Mathematics}, 475:110345, 2025.

\bibitem[Mei]{mei:18}
E.~Meinrenken.
\newblock Poisson geometry from a {D}irac perspective.
\newblock {\em Lett. Math. Phys.}, 108:447--498, 2018.

\bibitem[MeSa1]{melani-safranov-1}
V.~Melani and P.~Safronov.
\newblock Derived coisotropic structures {I}: affine case.
\newblock {\em Selecta Math. (N.S.)}, 24(4):3061--3118, 2018.

\bibitem[MeSa2]{melani-safranov-2}
V.~Melani and P.~Safronov.
\newblock Derived coisotropic structures {II}: stacks and quantization.
\newblock {\em Selecta Math. (N.S.)}, 24(4):3119--3173, 2018.

\bibitem[MoTa]{moo.tac:12}
G.~W. Moore and Y.~Tachikawa.
\newblock On $2d$ {TQFT}s whose values are holomorphic symplectic varieties.
\newblock {\em Proc. Sympos. Pure Math. Amer. Math. Soc.}, 85:191--207, 2012.

\bibitem[PTVV]{pan.toe.vaq.vez:13}
T.~Pantev, B.~To\"{e}n, M.~Vaqui\'{e}, and G.~Vezzosi.
\newblock Shifted symplectic structures.
\newblock {\em Publ. Math. Inst. Hautes \'{E}tudes Sci.}, 117:271--328, 2013.

\bibitem[Sev]{sev:11}
A.~Sevostyanov.
\newblock Algebraic group analogues of the {S}lodowy slices and deformations of
  {P}oisson ${W}$-algebras.
\newblock {\em Int. Math. Res. Not.}, 2011(8):1880--1925, 2011.

\bibitem[Ste]{ste:65}
R.~Steinberg.
\newblock Regular elements of semisimple algebraic groups.
\newblock {\em Publ. Math. IH\'{E}S}, 25:49--80, 1965.

\bibitem[\v{S}ev]{sev:15}
P.~\v{S}evera.
\newblock Left and right centers in quasi-{P}oisson geometry of moduli spaces.
\newblock {\em Adv. Math.}, 279:263--290, 2015.

\bibitem[\v{S}eWe]{sev.wei:01}
P.~\v{S}evera and A.~Weinstein.
\newblock Poisson geometry with a 3-form background.
\newblock In {\em Noncommutative geometry and string theory}, pages 145--154.

\bibitem[Xu]{xu:04}
P.~Xu.
\newblock Momentum maps and {M}orita equivalence.
\newblock {\em J. Differential Geom.}, 67(2):289--333, 2004.

\end{thebibliography}

\end{document}